\documentclass[11pt]{amsart}
\usepackage{amsmath, amssymb, amsthm, tikz-cd}
\usepackage{geometry,graphicx}
\newcommand{\ddc}{dd^c}
\newcommand{\C}{\mathbb{C}}
\newcommand{\R}{\mathbb{R}}
\newcommand{\Q}{\mathbb{Q}}
\newcommand{\Z}{\mathbb{Z}}

\newtheorem{lemma}{Lemma}
\newtheorem{thm}{Theorem}
\newtheorem{prop}{Proposition}
\newtheorem{cor}{Corollary}
\theoremstyle{definition}

\newtheorem{remark}{Remark}

\begin{document}

\title{On the Hodge Extension of Local Kähler Forms} 
\author{Rasul Shafikov}
\address{Department of Mathematics, University of Western Ontario, London, Ontario N6A 5B7, Canada}
\email{shafikov@uwo.ca}
\begin{abstract}
Let $X$ be a Stein manifold, and let $K=\{h\le0\}$ be a compact sublevel set
of a strictly plurisubharmonic function $h$ defined near $K$. It is known that
$K$ is meromorphically convex in $X$ if and only if $dd^ch$ extends to a Hodge
form on $X$. We show that if $H_2(X,\mathbb Z)$ is the direct sum of a free
abelian group and a torsion group, for example if it is finitely generated or if $X$
is a Stein surface, then any extension of $dd^ch$ to a K\"ahler form on $X$ 
can be modified, after replacing $h$ by a positive multiple, into a Hodge form. 
Thus, meromorphic convexity of $K$ is equivalent to the extendability of $dd^ch$ 
as a K\"ahler form. We also show that the topological assumption cannot be dropped: 
we construct a Stein manifold $X$ of dimension $6$ with
$H_2(X,\mathbb Z)\cong\mathbb Q$, a sublevel set $K$ as above that is not
meromorphically convex although $dd^ch$ extends to a K\"ahler form, and a
totally real torus in $K$ that is isotropic for a K\"ahler form but is not
meromorphically convex. This shows that the Hodge condition in Guedj's
characterization of meromorphically convex totally real submanifolds cannot
be weakened to a K\"ahler condition, and gives counterexamples to two criteria
stated by Cieliebak and Eliashberg.
\end{abstract}
\maketitle

\section{Introduction}

Let $X$ be a Stein manifold. By a complex hypersurface of $X$ we mean a closed
analytic subset of $X$ of pure codimension $1$. 
For a compact set $Q\subset X$, its \emph{meromorphically convex hull} is
\[
  \widehat Q_{\mathcal M(X)}=\{x\in X:\ \text{every complex hypersurface of $X$ through $x$ meets $Q$}\},
\]
and $Q$ is called \emph{meromorphically} (or \emph{hypersurface})
\emph{convex} if $\widehat Q_{\mathcal M(X)}=Q$. When $X=\mathbb C^n$, this gives the rationally 
convex hull of $Q$ (see~\cite{St}), and by~\cite[Prop. 1.2]{BS}, $  \widehat Q_{\mathcal M(X)}$ is the hull with respect to 
meromorphic functions on $X$.

It was shown by Nemirovski~\cite{Nem} that a compact sublevel set $K=\{ h \le 0\} \subset \mathbb C^n$ of a locally defined 
strictly plurisubharmonic function $h$ is rationally convex if and only if $dd^c h$ extends as a K\"ahler form on~$\mathbb C^n$. 
Similarly, by~\cite[Thm 1.2]{BGS},  a compact $K=\{ h \le 0\}$ on a Stein manifold $X$ is meromorphically 
convex  if and only if $dd^c h$ extends as a Hodge form on $X$, where by a Hodge form on a Stein manifold $X$ we 
mean a K\"ahler form $\Omega$ whose cohomology class $[\Omega] \in H^2(X, \R)$ lies in the image of the natural map 
$H^2(X, \Z) \to H^2(X, \R)$; equivalently, it has integral periods. This agrees with the result of Nemirovski because any 
K\"ahler form on $\mathbb C^n$ is exact, hence, Hodge. Compact sets of the form $\{h \le 0\}$ for locally defined strictly plurisubharmonic functions $h$ can be used to characterize meromorphic convexity of all compact sets on $X$, see~\cite[Thm 1.3]{BGS}.

On the other hand, Cieliebak and Eliashberg~\cite{CE} stated, without proof, that meromorphic convexity of $K$ on a Stein manifold $X$ (called rational convexity in~\cite{CE}) is equivalent to the extension of $dd^c h$ merely as a Kähler form on $X$. 
The following theorem confirms that, up to replacing $h$ by a positive multiple, which does not change $K$, all three conditions
are equivalent under a mild additional assumption on the topology of $X$.
Throughout the paper,  if $h$ is a function defined on an open set $U\supset K$, we
say that $dd^ch$ \emph{extends to a K\"ahler} (resp.\ \emph{Hodge}) \emph{form}
$\omega$ on $X$ if $\omega$ is a K\"ahler (resp.\ Hodge) form on $X$ with
$\omega=dd^ch$ on some neighbourhood of $K$. 

\begin{thm}\label{t.main}
Let $X$ be a Stein manifold such that
\begin{equation}\label{e.h2}
H_2(X, \mathbb{Z}) \cong \mathcal F \oplus \mathcal{T},
\end{equation}
where $\mathcal F$ is a free abelian group (possibly of infinite rank) and $\mathcal{T}$ is a torsion group.
Suppose that $U \subset X$ is an open set, $h: U \to \R$ is a strictly plurisubharmonic function, and $K=\{h \le 0 \}$ is a 
compact subset of $U$. Assume that $dd^c h$ extends to a Kähler form $\omega$ on $X$.
Then there exists a positive integer $N$ such that $\ddc (Nh)$ extends to a Hodge form $\Omega_H$ on $X$.
In particular, $K$ is meromorphically convex in $X$.
\end{thm}

Note that the decomposition~\eqref{e.h2} always holds if $H_2(X,\Z)$ is finitely generated, see Sec.~\ref{ssec:decomp} for details.
Further, by~\cite{AF}, a Stein manifold of dimension $n$ has the homotopy
type of a CW complex of dimension at most $n$. If $\dim_{\mathbb C}X=2$, the
cellular chain complex of such a CW complex has no $3$-chains, so
$H_2(X,\mathbb Z)$ is the group of cellular $2$-cycles, a subgroup of the free
abelian group of $2$-chains; hence, it is free. Thus, Theorem~\ref{t.main}
holds for every Stein surface. Also, Prop.~\ref{prop:exact} in Sec.~\ref{s.triv} shows that~\eqref{e.h2} is not
required if the image of $H_2$ of a neighbourhood of the holomorphically convex hull $\widehat K_{\mathcal O(X)}$ of $K$ in
$H_2(X,\Z)$ consists of torsion classes; in that case a K\"ahler extension
of $dd^ch$ can be taken exact.

But in general, as the next theorem shows, the topological hypothesis~\eqref{e.h2} in
Theorem~\ref{t.main} cannot be dropped. 

\begin{thm}\label{thm:counterexample}
There exist a Stein manifold $X$ of complex dimension $6$ with
$H_2(X,\mathbb{Z})\cong\mathbb{Q}$ and $\operatorname{Im}\bigl(H^2(X,\mathbb{Z})\to H^2(X,\mathbb{R})\bigr)=0$, an open set $U\subset X$, a strictly
plurisubharmonic function $h\colon U\to\mathbb{R}$ such that
$K=\{h\le 0\}$ is a compact subset of $U$, and a smooth compact totally real
torus $\Lambda\subset\operatorname{int}K$, such that the following hold.
\begin{enumerate}
\item[(i)] $dd^c h$ extends to a Kähler form $\omega$ on $X$.
\item[(ii)] The torus $\Lambda$ is isotropic for $\omega$, that is, $\omega|_\Lambda=0$.
\item[(iii)] For every $\lambda>0$, every K\"ahler form $\Omega$ on $X$
  extending $dd^c(\lambda h)$ satisfies $[\Omega]\neq 0$ in $H^2(X,\mathbb{R})$.
  In particular, for no positive integer $N$ does $dd^c(Nh)$ extend to a
  Hodge form on $X$.
  
\item[(iv)] The meromorphically convex hull $\widehat\Lambda_{\mathcal M(X)}$
  contains a holomorphic annulus $A$ with $\partial A\subset\Lambda$ and
  $A\not\subset K$. In particular, neither $\Lambda$ nor $K$ is
  meromorphically convex.
\end{enumerate}
\end{thm}

The theorem shows that the equivalence between meromorphic convexity of $K$ and the
extendability of $dd^c h$ as a Kähler form, stated in \cite[Criterion 1.6]{CE}, does not
hold on arbitrary Stein manifolds.

Concerning (ii) and (iv): a classical result of Duval and Sibony~\cite{DS}
states that a smooth compact totally real submanifold of $\mathbb{C}^n$ is
rationally convex  if and only if it is isotropic with respect to some Kähler form on
$\mathbb{C}^n$. Guedj~\cite{Gu} extended this to Stein manifolds: a smooth
compact totally real submanifold of a Stein manifold $X$ is meromorphically
convex if and only if it is isotropic with respect to some \emph{Hodge} form
on $X$, and in~\cite[Theorem~3.2]{BS} it was shown that isotropy with
respect to an exact Kähler form $dd^c\varphi$ is equivalent to strong meromorphic convexity 
(i.e., convexity with respect to principal hypersurfaces). In
$\mathbb{C}^n$ every Kähler form is exact, so these statements agree
with~\cite{DS}. By (ii) and (iv), the torus $\Lambda$ is isotropic for a
Kähler form but is not meromorphically convex, which gives a counterexample
to \cite[Criterion~1.8(a)]{CE}. By Guedj's theorem, $\Lambda$ is not
isotropic for any Hodge form on $X$. This can also be seen directly: every
Kähler form for which $\Lambda$ is isotropic has a positive period on an
explicit $2$-cycle (Lemma~\ref{lem:period}), while on $X$ the only integral
class is $0$.

The construction in Theorem~\ref{thm:counterexample} consists of two main steps.

\medskip

1. Construct a Stein manifold $Y$ with
\begin{equation}\label{e.Yh}
\operatorname{Hom}(H_1(Y,\mathbb{Z}),\mathbb{Z})=\operatorname{Hom}(H_2(Y,\mathbb{Z}),\mathbb{Z})=0
\quad\text{and}\quad H_1(Y,\mathbb{Z})\otimes\mathbb{Q}\neq0 .
\end{equation}
Such a manifold $Y_H$ of dimension $2$ was constructed by
Hulek~\cite[Prop.~2]{Hu}; it is the complement in the unit bidisc of a sequence
of smooth analytic curves (see also \cite[Example~1]{Co}). Our construction is
motivated by this example but is done differently: our $Y$ is a domain in
$\C^5$ that deformation retracts onto a domain in $\R^5$ with these
properties. By the Künneth formula, conditions~\eqref{e.Yh} ensure that every Hodge form on $X=Y\times\mathbb C^*$ is exact, while $H^2(X,\mathbb R)\neq0$.

2. In $Y\cap\R^5$ we place a smooth figure-eight curve $\Gamma$ whose
two loops represent a nonzero class in $H_1(Y,\mathbb{Q})$ and its negative. 
The torus $\Lambda\subset X$ lies over $\Gamma$; its fibres are
circles in $\mathbb{C}^*$ centred at $0$, with radius varying smoothly along
$\Gamma$ and taking different values on the two branches over the double point
 (see Fig.~\ref{fig:lambda} in Sec.~\ref{sec:counterexample}).
Since $\Gamma\subset\R^5$, which is totally real, the required
plurisubharmonic functions can be written down explicitly; this is used to
construct a global K\"ahler form and a local potential vanishing to second
order exactly on $\Lambda$. Although somewhat technical, this step uses only elementary
properties of plurisubharmonic functions, gluing with cutoff functions, and an
explicit correction by pluriharmonic functions near the double point. One then shows that 
$\widehat\Lambda_{\mathcal M(X)}$ contains a holomorphic annulus $A$ with $\partial A\subset\Lambda$.

\medskip

The mechanism producing the annulus $A$ --- a torus fibred over a figure-eight whose two lifts over the double point bound a holomorphic annulus in a fibre --- is the same as in the example of Duval and Gayet~\cite[Example]{DG} (see also~\cite{Al}) of a totally real torus in $S^3\subset\mathbb C^2$ that is not rationally convex. In $\mathbb C^2$ such a torus cannot be isotropic for any Kähler form by~\cite{DS}. On our $X$ there are non-exact
K\"ahler forms, while every Hodge form is exact; this is what allows
$\Lambda$ to be isotropic for a K\"ahler form but not for a Hodge form.

The same construction can be carried out in Hulek's domain $Y_H$, which gives,
with $X=Y_H\times\C^*$, an example as in Theorem~\ref{thm:counterexample} with
$\dim_{\mathbb C}X=3$. Without going into details the idea is the following: $Y_H\subset\mathbb{C}^2$ satisfies the
topological conditions above, and a figure-eight $\Gamma\subset Y_H$ with the
required homology, straight near its double point, can be made
$\mathcal{O}(Y_H)$-convex by a small perturbation away from the double point.
A nonnegative strictly plurisubharmonic function on $Y_H$ vanishing on
$\Gamma$ to second order is then obtained by gluing a local model near
$\Gamma$ with a global strictly plurisubharmonic function, using the
$\mathcal{O}(Y_H)$-convexity of $\Gamma$, and the rest of the argument in
Sec.~\ref{sec:counterexample} applies with only notational changes. We
chose to present the higher-dimensional version because there the
construction of the local potential is particularly explicit and elementary,
and no perturbation of $\Gamma$ is needed. Since Theorem~\ref{t.main} holds
for all Stein surfaces, the dimension $3$ is optimal.

\medskip

\subsection*{Acknowledgments}
The author obtained the proof of Theorem~\ref{t.main} in the fall of 2025,
but all attempts to remove assumption~\eqref{e.h2} were unsuccessful. The
counterexample of Theorem~\ref{thm:counterexample} was eventually found with
the help of AI, specifically Claude Fable~5.1, which proposed
the original construction; the construction presented here has been
substantially simplified. AI was also used to assist with
editing the exposition. The author bears full responsibility for the
correctness of all proofs and for the content of this paper. The author thanks Purvi Gupta for
many valuable discussions related to the problem, and in particular for
identifying $H_2(X,\mathbb Z)\cong\mathbb Q$ as a potential source of a
counterexample. The author is partially supported by the Natural Sciences and
Engineering Research Council of Canada.

\section{Background: periods, Hodge classes, Levi form}\label{s.top}

Sections~\ref{ssec:periods} and~\ref{ssec:decomp} collect in one place the facts relating homology and cohomology
that are used throughout the paper: the universal coefficient theorem, periods
of closed forms, integral (Hodge) classes, and the structure of $H_2(X,\Z)$.
The results are standard. Most of the material can be found in
\cite[Section~3.1]{Hatcher} for the universal coefficient theorem and the
properties of $\operatorname{Ext}$, and in \cite[Chapter~18]{Lee} for the
de Rham theorem in terms of integration over smooth singular chains. We include
short proofs to fix the notation. Sec.~\ref{ssec:levi} discusses the Levi form with some 
estimates that will be used in the proof of the main results.

In the next two subsections $X$ is an arbitrary smooth manifold.

\subsection{Periods and the universal coefficient theorem}\label{ssec:periods}

If $\beta$ is a closed $k$-form on $X$ and $C$ is an integral $k$-cycle, the
period $\int_C\beta$ depends only on the de Rham class $[\beta]$ and on the
homology class $[C]\in H_k(X,\Z)$. Hence, a class
$c\in H^k(X,\R)$ (or $H^k(X,\C)$) defines a homomorphism
$\langle c,\cdot\rangle\colon H_k(X,\Z)\to\R$ (or $\C$), the \emph{period
homomorphism} of $c$. It vanishes on torsion classes, because $\R$ and $\C$
are torsion-free. 

For every abelian group $A$ the universal coefficient theorem
gives a short exact sequence
\begin{equation}\label{eq:UCT}
0\to\operatorname{Ext}\bigl(H_{k-1}(X,\Z),A\bigr)\to H^k(X,A)\to
\operatorname{Hom}\bigl(H_k(X,\Z),A\bigr)\to0,
\end{equation}
which is natural in $A$ and in $X$.

\begin{prop}\label{p.UCT}
We have the following.
\begin{enumerate}
\item[(a)] The period map
  $H^2(X,\R)\to\operatorname{Hom}(H_2(X,\Z),\R)$ is an isomorphism. In particular,
  a closed real $2$-form on $X$ is exact if and only if all its periods on
  integral $2$-cycles vanish, two closed real $2$-forms are cohomologous if and
  only if they have the same periods, and every homomorphism
  $H_2(X,\Z)\to\R$ is the period homomorphism of a unique class in
  $H^2(X,\R)$. Moreover, $H^2(X,\Q)\to H^2(X,\R)$ is injective, with image the
  set of classes whose periods are rational.
\item[(b)] Let $\mathcal H_X$ be the image (the \emph{Hodge lattice}) of
  $H^2(X,\Z)\to H^2(X,\R)$. A class $c\in H^2(X,\R)$ lies in $\mathcal H_X$ if and
  only if $\langle c,x\rangle\in\Z$ for every $x\in H_2(X,\Z)$. Thus the period
  map identifies $\mathcal H_X$ with $\operatorname{Hom}(H_2(X,\Z),\Z)$.
\item[(c)] If $\operatorname{Hom}(H_2(X,\Z),\Z)=0$, then $\mathcal H_X=0$.
\end{enumerate}
\end{prop}

\begin{proof}
(a) The groups $\R$ and $\Q$ are divisible, hence injective $\Z$-modules, so
$\operatorname{Ext}(H_1(X,\Z),A)=0$ for $A=\R,\Q$, and~\eqref{eq:UCT} with $k=2$
gives isomorphisms $H^2(X,A)\cong\operatorname{Hom}(H_2(X,\Z),A)$. For $A=\R$,
composed with the de Rham isomorphism, this map is the period map, which gives
the first three statements. By naturality of~\eqref{eq:UCT} in $A$, the map
$H^2(X,\Q)\to H^2(X,\R)$ induced by the inclusion $\Q\subset\R$ corresponds to
$\operatorname{Hom}(H_2(X,\Z),\Q)\to\operatorname{Hom}(H_2(X,\Z),\R)$, which is
injective with image the $\Q$-valued homomorphisms. This gives the last
statement.

(b) If $c$ is the image of $c'\in H^2(X,\Z)$, then
$\langle c,x\rangle=\langle c',x\rangle\in\Z$ by naturality in $A$. Conversely,
if all periods of $c$ are integers, then $\chi=\langle c,\cdot\rangle$ lies in
$\operatorname{Hom}(H_2(X,\Z),\Z)$. By the surjectivity in~\eqref{eq:UCT} with
$A=\Z$, there is $c'\in H^2(X,\Z)$ whose period homomorphism is $\chi$. Its image
in $H^2(X,\R)$ also has period homomorphism $\chi$, so it equals $c$ by (a).

(c) follows from (b).
\end{proof}

A \emph{Hodge form} on a complex manifold $X$ is a K\"ahler form $\Omega$ with
$[\Omega]\in\mathcal H_X$. By Prop.~\ref{p.UCT}(b), $\Omega$ is a Hodge form
if and only if all its periods on integral $2$-cycles are integers. By
Prop.~\ref{p.UCT}(c), if $\operatorname{Hom}(H_2(X,\Z),\Z)=0$ then a
K\"ahler form is a Hodge form if and only if it is exact; on a Stein manifold this
means that it is of the form $dd^c\varphi$ with $\varphi$ strictly plurisubharmonic
(Remark~\ref{rem:exact}).

\subsection{The decomposition of $H_2(X,\Z)$}\label{ssec:decomp}

Suppose that $H_2(X,\Z)=\mathcal F\oplus\mathcal T$ as in~\eqref{e.h2}, with
$\mathcal F$ a free abelian group (of possibly infinite rank) and $\mathcal T$
a torsion group. This holds, for instance, whenever $H_2(X,\Z)$ is finitely
generated, by the structure theorem for finitely generated abelian groups. Since $\mathcal F$ is
torsion-free, $\mathcal T$ is the torsion subgroup of $H_2(X,\Z)$. The
subgroup $\mathcal F$ is not canonical, but we fix one such choice, and let
$\operatorname{pr}\colon H_2(X,\Z)\to\mathcal F$ be the projection with kernel
$\mathcal T$. Any homomorphism from $H_2(X,\Z)$ to a torsion-free group (such as
$\Z,\Q,\R$) vanishes on $\mathcal T$, so it is determined by its restriction
to $\mathcal F$, and $\langle c,x\rangle=\langle c,\operatorname{pr}(x)\rangle$.
Prop.~\ref{p.UCT} therefore gives the following.

\begin{cor}\label{c.F}
If $H_2(X,\Z)=\mathcal F\oplus\mathcal T$ as above, then restriction to
$\mathcal F$ identifies $H^2(X,\R)$ with $\operatorname{Hom}(\mathcal F,\R)$,
the lattice $\mathcal H_X$ with $\operatorname{Hom}(\mathcal F,\Z)$, and the image
of $H^2(X,\Q)$ with $\operatorname{Hom}(\mathcal F,\Q)$. If $\mathcal F$ has finite
rank $b$ and $C_1,\dots,C_b$ are integral $2$-cycles whose classes form a basis
of $\mathcal F$, then the period map
\[
[\beta]\longmapsto\Big(\int_{C_1}\beta,\dots,\int_{C_b}\beta\Big)
\]
is an isomorphism $H^2(X,\R)\cong\R^b$ that maps $\mathcal H_X$ onto $\Z^b$ and
the rational classes onto $\Q^b$. In particular, a real class is Hodge if and
only if its periods on $C_1,\dots,C_b$ are integers, and every rational class
becomes integral after multiplication by a common denominator of its periods.
\end{cor}

We shall also need two algebraic lemmas. Both are used for
infinitely generated $\mathcal F$, where the finitely generated pieces must be
controlled by hand.

\begin{lemma}\label{l.sat}
Let $\mathcal F$ be a free abelian group, let $I\subset\mathcal F$ be a finitely
generated subgroup, and let
$\mathcal S=\{x\in\mathcal F:\ nx\in I\ \text{for some } n\ge1\}$
be the {\it saturation} of $I$. Then
$\mathcal S$ is a finitely generated direct summand of $\mathcal F$, and
$\mathcal F/\mathcal S$ is free.
\end{lemma}

\begin{proof}
Fix a basis of $\mathcal F$. Since $I$ is finitely generated, it lies in the span
$\mathcal F_M$ of finitely many basis elements, and
$\mathcal F=\mathcal F_M\oplus\mathcal F_M'$, where $\mathcal F_M'$ is spanned by
the remaining basis elements. If $nx\in\mathcal F_M$ with $n\ne0$, the
$\mathcal F_M'$-component of $x$ vanishes, so $\mathcal S\subset\mathcal F_M$.
Thus $\mathcal F_M/\mathcal S$ is finitely generated and torsion-free, hence
free. Therefore $\mathcal S$ is a direct summand of $\mathcal F_M$ with finitely
generated (free) complement, so $\mathcal S$ is finitely generated, and it is a
direct summand of $\mathcal F$ with $\mathcal F/\mathcal S\cong\mathcal F_M/\mathcal S\oplus\mathcal F_M'$ free.
\end{proof}

\begin{lemma}\label{l.pre}
Let $H_2(X,\Z)=\mathcal F\oplus\mathcal T$ as above, let $\mathcal F=\mathcal S\oplus\mathcal B$ be a
decomposition of $\mathcal F$, and let $\psi\colon\mathcal S\to\R$ be a
homomorphism. Define $\chi_\psi\colon H_2(X,\Z)\to\R$ by
\[
  \chi_\psi(s+b+t)=\psi(s),\qquad s\in\mathcal S,\ b\in\mathcal B,\ t\in\mathcal T .
\]
Then there is a unique class $c_\psi\in H^2(X,\R)$ with
period homomorphism $\chi_\psi$. Its periods are $\psi$ on $\mathcal S$ and $0$ on
$\mathcal B\oplus\mathcal T$. If $N\ge1$ is an integer with
$N\psi(\mathcal S)\subset\Z$, then $Nc_\psi\in\mathcal H_X$.
\end{lemma}

\begin{proof}
Since $H_2(X,\Z)=\mathcal S\oplus\mathcal B\oplus\mathcal T$, the map $\chi_\psi$ is a
well-defined homomorphism, and the class $c_\psi$ exists and is unique by
Prop.~\ref{p.UCT}(a). If $N\psi(\mathcal S)\subset\Z$, then $N\chi_\psi$ takes integer
values, so $Nc_\psi\in\mathcal H_X$ by Prop.~\ref{p.UCT}(b).
\end{proof}

\subsection{The Levi form}\label{ssec:levi}

For a real $C^2$ function $f$ on an open subset of $\R^m$, we write
$\operatorname{Hess}_xf$ for its real Hessian at $x$, that is, the symmetric
bilinear form $\operatorname{Hess}_xf(v,w)=\sum_{j,k}\frac{\partial^2f}{\partial x_j\partial x_k}(x)\,v_jw_k$,
and $\|\operatorname{Hess}_xf\|$ for its operator norm. In general
$\operatorname{Hess}_xf$ depends on the coordinates. However, if $f$ is a real
$C^2$ function on a manifold $X$ and $df(q)=0$, then $\operatorname{Hess}_qf$
is independent of the chart, and we regard it as a quadratic form on $T_qX$.

The rest of this subsection is local, so we work in $\C^n$ with coordinates
$z=x+iy$, $x,y\in\R^n$, and regard $\C^n$ as $\R^{2n}$. We identify
$\xi=a+ib\in\C^n$ with the real vector $v=(a,b)\in\R^{2n}$, so that
$Jv=(-b,a)$ corresponds to $i\xi$. The Levi form of a real $C^2$ function $u$
at $z$ is
\[
  L_u(z;\xi)=\sum_{j,k=1}^n\frac{\partial^2u}{\partial z_j\partial\bar z_k}(z)\,\xi_j\bar\xi_k .
\]
It is related to the real Hessian and to $dd^c$ by
\begin{equation}\label{eq:levi}
  L_u(z;\xi)=\tfrac14\bigl(\operatorname{Hess}_zu(v,v)+\operatorname{Hess}_zu(Jv,Jv)\bigr)=dd^cu(v,Jv) ,
\end{equation}
where  $d^c=\frac{i}{4}(\bar\partial-\partial)$, so that $dd^c=d\circ d^c = \frac{i}{2} \partial \bar{\partial}$.
Indeed, put $g(\lambda)=u(z+\lambda\xi)$ for $\lambda=t+is\in\C$ near $0$.
By the chain rule $\partial_\lambda\partial_{\bar\lambda}g(0)=L_u(z;\xi)$. On
the other hand $\partial_\lambda\partial_{\bar\lambda}=\frac14(\partial_t^2+\partial_s^2)$,
and $\partial_t^2g(0)=\operatorname{Hess}_zu(v,v)$,
$\partial_s^2g(0)=\operatorname{Hess}_zu(Jv,Jv)$, since $\partial_t$ and
$\partial_s$ differentiate in the directions $v$ and $Jv$. This gives the
first equality. For the second, we have
$dd^cu=\frac i2\sum_{j,k}\frac{\partial^2u}{\partial z_j\partial\bar z_k}\,dz_j\wedge d\bar z_k$,
and $dz_j\wedge d\bar z_k(v,Jv)=\xi_j\cdot\overline{i\xi_k}-i\xi_j\cdot\bar\xi_k=-2i\,\xi_j\bar\xi_k$.
In particular, $u$ is strictly plurisubharmonic if and only if
$L_u(z;\xi)>0$ for all $z$ and all $\xi\ne0$.

For example, if $u(x+iy)=F(x)+\sigma|y|^2$ with $F\in C^2(\R^n)$ and
$\sigma\in\R$, then $\operatorname{Hess}_zu$ is block diagonal, with blocks
$\operatorname{Hess}_xF$ and $2\sigma\,\mathrm{Id}$, and~\eqref{eq:levi} gives
\begin{equation}\label{eq:levi-example}
  L_u(z;\xi)=\tfrac14\Bigl(\operatorname{Hess}_xF(a,a)+\operatorname{Hess}_xF(b,b)+2\sigma|\xi|^2\Bigr)
  \ \ge\ \tfrac14\bigl(2\sigma-\|\operatorname{Hess}_xF\|\bigr)|\xi|^2 .
\end{equation}

Finally, we record two product formulas, which hold on any complex manifold.
For real $C^2$ functions $f$ and $g$, we have $d^c(fg)=f\,d^cg+g\,d^cf$, and
applying $d$ gives
\begin{equation}\label{eq:ddc-product}
  dd^c(fg)=f\,dd^cg+g\,dd^cf+df\wedge d^cg+dg\wedge d^cf .
\end{equation}
For $f=g=v$, since
$dv\wedge d^cv=\frac i4(\partial v+\bar\partial v)\wedge(\bar\partial v-\partial v)
=\frac i2\,\partial v\wedge\bar\partial v$, this becomes
\begin{equation}\label{eq:ddc-square}
  dd^c(v^2)=i\,\partial v\wedge\bar\partial v+2v\,dd^cv .
\end{equation}

\section{Preliminary Results on Kähler Forms}

In this section, we prove some preliminary results concerning K\"ahler forms and their cohomology classes.
Lemmas~\ref{l.1} and~\ref{l.3} are well-known, but proofs are hard to find. We provide them below for the sake of
completeness.

\begin{lemma}\label{l.1}
Let $X$ be a Stein manifold.
\begin{enumerate}
\item[(a)] Every class in $H^2(X,\mathbb R)$ is represented by a smooth,
  closed, real $(1,1)$-form on $X$.
\item[(b)] Let $D\subset X$ be a compact $\mathcal O(X)$-convex set and let $w$
  be a smooth, closed, real $(1,1)$-form on $X$ which is positive definite at
  every point of $D$. Then there is a K\"ahler form $\omega$ on $X$ such that
  $[\omega]=[w]$ in $H^2(X,\mathbb R)$ and $\omega=w$ on $D$.
\end{enumerate}
In particular, taking $D=\emptyset$ in (b), every class in $H^2(X,\mathbb R)$
is represented by a K\"ahler form.
\end{lemma}

\begin{proof}
(a) Let $u$ be a smooth, closed, real $2$-form representing the given class,
and decompose it into types, $u=u^{2,0}+u^{1,1}+u^{0,2}$. The
$(0,3)$-component of $du=0$ is $\bar\partial u^{0,2}$, so $u^{0,2}$ is
$\bar\partial$-closed. Since $X$ is Stein, $H^{0,2}(X)=0$, so there is a
$(0,1)$-form $p$ with $\bar\partial p=u^{0,2}$. Put $\gamma=p+\bar p$ and
$w=u-d\gamma$. The $(0,2)$-component of $d\gamma$ is $\bar\partial p=u^{0,2}$,
so $w^{0,2}=0$. Since $w$ is real, $w^{2,0}=\overline{w^{0,2}}=0$. Thus $w$ is
a smooth, closed, real $(1,1)$-form cohomologous to $u$.

(b) Since $w$ is positive definite at
the points of the compact set $D$, it is positive definite on an open
neighbourhood $O$ of $D$. As $D$ is $\mathcal O(X)$-convex, there is a smooth
strictly plurisubharmonic exhaustion function $\Psi$ on $X$ with $\Psi<0$ on
$D$ and $\Psi>0$ on $X\setminus O$ \cite[Theorem~5.1.6]{Hormander}. (If
$D=\emptyset$, take $O=\emptyset$ and any strictly plurisubharmonic exhaustion
function, to which we add a constant so that $\Psi>0$.) The compact sets
$\{\Psi\le1/\lambda\}$ decrease, as $\lambda\to\infty$, to $\{\Psi\le0\}\subset O$,
so after replacing $\Psi$ by $\lambda\Psi$ with $\lambda>0$ large, we may
assume that
\[
  \{\Psi\le1\}\subset O .
\]
For $j\ge1$ put
\[
  \nu_j:=\max\Bigl\{\frac{-w(v,Jv)}{dd^c\Psi(v,Jv)}:\ x\in\{\Psi\le j+1\},\
  0\ne v\in T_xX\Bigr\}.
\]
This is finite, because $\{\Psi\le j+1\}$ is compact and $dd^c\Psi(v,Jv)>0$ for
$v\ne0$. By definition, for every $x\in\{\Psi\le j+1\}$ and every $s>\nu_j$ the
form $w+s\,dd^c\Psi$ is positive definite at $x$.
Choose a nondecreasing sequence $s_j>\max\{\nu_j,0\}$, and let $G$ be the
piecewise linear function with slope $0$ on $(-\infty,0]$, slope $s_1$ on
$[0,2]$, and slope $s_j$ on $[j,j+1]$ for $j\ge2$. Its slopes are nonnegative
and nondecreasing. Smooth $G'$ to a nondecreasing smooth function by
rounding the corner at $0$ in a small interval to its right, and each corner
at $j\ge2$ in a small interval to its left. We obtain a smooth convex function
$G\colon\mathbb R\to\mathbb R$ with
\[
  G=0\ \text{on}\ (-\infty,0],\qquad G'\ge0,\qquad
  G'>\nu_j\ \text{on}\ [j,j+1]\ \ (j\ge1).
\]
Put $\omega=w+dd^c(G\circ\Psi)$. This is a
smooth, closed, real $(1,1)$-form, and
\[
  dd^c(G\circ\Psi)=G'(\Psi)\,dd^c\Psi+G''(\Psi)\,d\Psi\wedge d^c\Psi
  \ \ge\ G'(\Psi)\,dd^c\Psi\ \ge\ 0 ,
\]
since $G'\ge0$ and $G''\ge0$. We check that $\omega$ is positive definite
at every $x\in X$: if $\Psi(x)\le1$, then $x\in O$, where $w>0$, so $\omega\ge w>0$ at $x$;
if $\Psi(x)\in[j,j+1]$ for some $j\ge1$, then $x\in\{\Psi\le j+1\}$ and
  $G'(\Psi(x))>\nu_j$, so at $x$
  \[
    \omega\ \ge\ w+G'(\Psi(x))\,dd^c\Psi\ >\ 0 .
  \]

Hence, $\omega$ is a Kähler form, which is clearly cohomologous to $w$. Finally,
$G\circ\Psi=0$ on the open set $\{\Psi<0\}\supset D$, so $\omega=w$ there, in
particular on $D$.
\end{proof}

\begin{lemma}[$dd^c$-Lemma]\label{l.3}
Let $X$ be a Stein manifold.
\begin{enumerate}
    \item[(a)] Let $\eta$ be a smooth, real $(1,1)$-form on $X$. If $\eta$ is $d$-exact, then there exists a smooth function 
    $h \in C^\infty(X)$ such that $dd^c h = \eta$.
    
\item[(b)] Fix a Hermitian metric on $X$ and let $n=\dim_{\mathbb C}X$. Let
$U\Subset X$ be a strictly pseudoconvex domain with smooth boundary. If $\eta$
is a smooth, real $(1,1)$-form on $X$ which is $d$-exact on a neighbourhood of 
$\overline U$, then there
exists $h\in C^\infty(\overline U)$ such that $dd^c h=\eta$ on $U$ and
\begin{equation}\label{eq:ddc-estimate}
  \|h\|_{C^2(U)} \le C\,\|\eta\|_{C^{n+2}(U)},
\end{equation}
where $C$ depends on $U$ and the metric but not on $\eta$.
\end{enumerate}
\end{lemma}

\begin{proof}
(a) Since $\eta$ is  $d$-exact, there exists a smooth real 1-form $\xi=\xi^{1,0} + \xi^{0,1}$ such that $d\xi = \eta$. The $(0,2)$-component of $d\xi$ is $\bar{\partial} \xi^{0,1}$. Since $\eta$ is of type $(1,1)$, we have $\bar{\partial} \xi^{0,1} = 0$. Since $X$ is Stein, $H^{0,1}(X) = 0$, so there exists a smooth function $f$ such that $\bar{\partial} f = \xi^{0,1}$. Then, using that $\eta$ is real, 
$$\eta = d\xi = \partial \xi^{0,1} + \bar{\partial} \xi^{1,0} = \partial(\bar{\partial} f) + \bar{\partial}(\partial \bar{f}) = \partial \bar{\partial} (f - \bar{f}).$$
Let $h = 4\text{Im}(f)$. Then $f - \bar{f} = 2i \text{Im}(f) = \frac{i}{2} h$. Thus, $\eta = \partial \bar{\partial} (\frac{i}{2} h) = \frac{i}{2} \partial \bar{\partial} h = dd^c h$.

\medskip
(b) Write $\|\cdot\|_s$ for the Sobolev $W^s$-norm on $U$ with respect to the
fixed metric. By~\cite[Lemma~3.2.1]{Sch},
applied on the compact manifold with boundary $\overline U$, $\eta$ admits a
primitive $\xi\in C^\infty(\overline U)$, $d\xi=\eta$, with
\[
  \|\xi\|_{s}\le C_1\|\eta\|_{s-1}\qquad (s\ge 1).
\]
Replacing $\xi$ by its real part, we may assume $\xi$ is real. As in (a),
$\bar\partial\xi^{0,1}=0$.

Since $U$ is strictly pseudoconvex, by H\"ormander's $L^2$-estimates every
$\bar\partial$-closed $(0,1)$-form in $L^2(U)$ is $\bar\partial$-exact in
$L^2(U)$. Hence, there are no nonzero $\bar\partial$-harmonic $(0,1)$-forms
on $U$, and in particular the harmonic projection $H_{(0,1)}\xi^{0,1}$
vanishes. By \cite[Corollary~5.3.11]{CS}, there exists
$f\in C^\infty(\overline U)$ with $\bar\partial f=\xi^{0,1}$ and
\[
  \|f\|_{s+\frac12}\le C_2\|\xi^{0,1}\|_{s}.
\]
As in (a), $h:=4\operatorname{Im}(f)\in C^\infty(\overline U)$ is real and
satisfies $dd^c h=\eta$ on $U$.

Now take $s=n+3$. Since $U$ has real dimension $2n$ and smooth boundary, the
Sobolev embedding $W^{s}(U)\hookrightarrow C^2(\overline U)$ holds for
$s>2+n$. Therefore,
\[
  \|h\|_{C^2(U)}\le 4\|f\|_{C^2(U)}\le C_3\|f\|_{s}\le C_3C_2\|\xi^{0,1}\|_{s}
  \le C_4\|\xi\|_{s}\le C_5\|\eta\|_{s-1}\le C\,\|\eta\|_{C^{n+2}(U)},
\]
where the last step uses that $U$ has finite volume.
\end{proof}

\begin{remark}\label{rem:exact}
By Lemma~\ref{l.3}(a), an exact K\"ahler form $\omega$ on a Stein manifold $X$
is of the form $\omega=dd^c\varphi$ with $\varphi\in C^\infty(X)$; conversely,
$dd^c\varphi=d(d^c\varphi)$ is exact. Thus, for K\"ahler forms on Stein
manifolds, ``exact'' and ``$dd^c$-exact'' mean the same thing.
\end{remark}

Next, we prove the following.

\begin{lemma}\label{l:loc}
Let $X$ be a Stein manifold, let $U\subset X$ be a Stein open subset, and let
$K\subset U$ be compact. Let $\varphi\colon H_2(X,\mathbb Z)\to\mathbb R$ be a
homomorphism which vanishes on the image of the map
$H_2(U,\mathbb Z)\to H_2(X,\mathbb Z)$ induced by inclusion. Then there exists
a smooth, closed, real $(1,1)$-form $\beta$ on $X$ such that
\begin{enumerate}
  \item $\int_S\beta=\varphi(S)$ for every $S\in H_2(X,\mathbb Z)$;
  \item $\beta=0$ on a neighbourhood of $K$.
\end{enumerate}
\end{lemma}

\begin{proof}
By Prop.~\ref{p.UCT}(a) there is a class in $H^2(X,\mathbb R)$ whose
periods are given by $\varphi$. By Lemma~\ref{l.1}(a), it is represented
by a smooth, closed, real $(1,1)$-form $\alpha$ on $X$. Every integral $2$-cycle $S'$ in $U$ satisfies
$\int_{S'}\alpha=\varphi([S'])=0$, since $[S']$ lies in the image of
$H_2(U,\mathbb Z)$. Therefore, $\alpha|_U$ is exact, by
Prop.~\ref{p.UCT}(a) applied to $U$. Since $U$ is Stein,
Lemma~\ref{l.3}(a) gives $f\in C^\infty(U)$ with $dd^cf=\alpha$ on $U$.

Choose a neighbourhood $W$ of $K$ and a smooth function $\rho\colon X\to[0,1]$
with $\rho\equiv1$ on $W$ and $\operatorname{supp}\rho\Subset U$. Then
$\rho f$, extended by zero, is smooth on $X$, and
\[
  \beta=\alpha-dd^c(\rho f)
\]
is a smooth, closed, real $(1,1)$-form on $X$. Since $dd^c(\rho f)$ is exact,
$\beta$ has the same periods as $\alpha$, which proves (1). On $W$ we have
$\beta=\alpha-dd^cf=0$, which proves (2).
\end{proof}

Finally, we need the following approximation result.

\begin{lemma}[Rational approximation]\label{l.ra}
Let $X$ be a Stein manifold such that $H_2(X,\mathbb Z)$ is finitely
generated. Fix a Hermitian metric on $X$ and a K\"ahler form $\omega$ on $X$.
Let $\Sigma\subset H_2(X,\mathbb Z)$ be a set 
such that $\int_S\omega=0$ for all $S\in\Sigma$. Then for every
neighbourhood $\mathcal U$ of $[\omega]$ in the finite-dimensional space
$H^2(X,\mathbb R)$, every integer $k\ge 0$, every $\epsilon>0$ and every
compact set $D\subset X$ there exists a smooth, closed, real $(1,1)$-form
$\tilde\omega_D$ on $X$ such that
\begin{enumerate}
  \item[(a)] $[\tilde\omega_D]\in H^2(X,\mathbb Q)\cap\mathcal U$ and
        $\int_S\tilde\omega_D=0$ for all $S\in\Sigma$;
  \item[(b)] $\tilde\omega_D$ is positive definite at every point of $D$ and
        $\|\tilde\omega_D-\omega\|_{C^k(D)}<\epsilon$.
\end{enumerate}
\end{lemma}

\begin{proof}
Since $H_2(X,\Z)$ is finitely generated, we may write
$H_2(X,\Z)=\mathcal F\oplus\mathcal T$ as in Sec.~\ref{ssec:decomp}, with
$\mathcal F\cong\Z^b$. Let $C_1,\dots,C_b$ be integral $2$-cycles whose
classes form a basis of $\mathcal F$. By Prop.~\ref{p.UCT}(a) and
Lemma~\ref{l.1}(a), there are real closed $(1,1)$-forms $e_1,\dots,e_b$ on $X$ with
\[
  \int_{C_l}e_j=\delta_{jl},\qquad 1\le j,l\le b .
\]
We identify $H^2(X,\mathbb R)$ with $\mathbb R^b$ via the period map
$[\beta]\mapsto x=(\int_{C_1}\beta,\dots,\int_{C_b}\beta)$ of Corollary~\ref{c.F}; under this
identification $H^2(X,\mathbb Q)=\mathbb Q^b$. Each $S\in\Sigma$ can be written
as $S=\sum_j a_j(S)\,[C_j]+t$ with $a_j(S)\in\mathbb Z$ and $t\in \mathcal T$. Since
periods vanish on torsion classes, $\int_S \beta =\sum_j a_j(S) \int_{C_j} \beta$, where
$a(S)=(a_1(S),\dots,a_b(S))\in\mathbb Z^b$.

Let $a^1,\dots,a^r\in\mathbb Z^b$ be a
basis of the $\mathbb Q$-span of $\{a(S):S\in\Sigma\}$, and let $A$ be the
$r\times b$ integer matrix with rows $a^1,\dots,a^r$. Then the classes
vanishing on $\Sigma$ form the subspace
$L=\ker_{\mathbb R}A\subset\mathbb R^b$. The rank of $A$ is the same over
$\mathbb Q$ and over $\mathbb R$, so $\ker_{\mathbb Q}A$ has
$\mathbb Q$-dimension $b-\operatorname{rank}A=\dim_{\mathbb R}L$, and a
$\mathbb Q$-basis of $\ker_{\mathbb Q}A$ is an $\mathbb R$-basis of $L$.
Hence, $L\cap\mathbb Q^b$ is dense in $L$.

Write $\lambda_j=\int_{C_j}\omega$, so that
$\lambda=(\lambda_1,\dots,\lambda_b)\in L$ by hypothesis. Fix $\mathcal U$,
$k$, $\epsilon$, and $D$. Since $\omega$ is positive definite and $D$ is
compact, there is $\epsilon_0>0$ such that every real $(1,1)$-form $\alpha$
with $\|\alpha-\omega\|_{C^0(D)}<\epsilon_0$ is positive definite at every
point of $D$. Let $M_D=\max\{1,\max_j\|e_j\|_{C^k(D)}\}$, and choose
$\delta\in\bigl(0,\min\{\epsilon,\epsilon_0\}/(bM_D)\bigr]$ so small that every
$x\in\mathbb R^b$ with $|x_j-\lambda_j|<\delta$ for all $j$ lies in
$\mathcal U$. By density, there is $q\in L\cap\mathbb Q^b$ with
$|\lambda_j-q_j|<\delta$ for all $j$.

Define
\[
  \tilde\omega_D=\omega-\sum_{j=1}^b(\lambda_j-q_j)\,e_j .
\]
This is a smooth, closed, real $(1,1)$-form with
$\int_{C_l}\tilde\omega_D=\lambda_l-(\lambda_l-q_l)=q_l$ for every $l$. Thus,
the class $[\tilde\omega_D]$ corresponds to $q\in L\cap\mathbb Q^b$ under the
period map, so it is rational, lies in $\mathcal U$ by the choice of $\delta$,
and satisfies $\int_S\tilde\omega_D=a(S)\cdot q=0$ for all $S\in\Sigma$. This
proves (a). Moreover
\[
  \|\tilde\omega_D-\omega\|_{C^k(D)}\le\sum_{j=1}^b|\lambda_j-q_j|\,
  \|e_j\|_{C^k(D)}<b\,\delta\,M_D\le\min\{\epsilon,\epsilon_0\}.
\]
Since $\|\cdot\|_{C^0(D)}\le\|\cdot\|_{C^k(D)}$, the form $\tilde\omega_D$ is positive
definite on $D$, which proves (b).
\end{proof}

\section{Proof of Theorem~\ref{t.main}}

We first give the proof for finite topology, then extend it to the general case.

\subsection{Case 1: $H_2(X,\Z)$ is finitely generated}
Let $n=\dim_{\mathbb C}X$ and fix a Hermitian metric on $X$; all $C^k$-norms
below are taken with respect to it. Shrinking $U$ to a smaller neighbourhood
of $K$ does not change $K=\{h\le0\}$, so we may assume that $\omega=dd^ch$
on $U$.

Choose an open set $U_1$ with $K\subset U_1\Subset U$; then
$m_1:=\min_{\partial U_1}h>0$, since $h>0$ on $U\setminus K$. By Sard's theorem we
may choose regular values $0<c<c'<m_1$ of $h$, and set
\[
  V=\{x\in U_1: h(x)<c\},\qquad V'=\{x\in U_1: h(x)<c'\}.
\]
Then $K\subset V\Subset V'\Subset U$, $\omega=dd^ch$ on $V'$, and $V$ is a
strictly pseudoconvex domain with smooth boundary $\{h=c\}$.
Let $\Psi$ be a smooth strictly plurisubharmonic exhaustion function on $X$; 
choose $R\in\mathbb R$ with $\overline{V'}\subset\{\Psi<R\}$ and set
$D=\{\Psi\le R\}$. Then $D$ is compact, and
\[
\overline V\subset\overline{V'}\subset\operatorname{int}D
\subset D .
\]
Let $\epsilon>0$; its value will be fixed later. It depends only on $V$,
$D$, the constant of Lemma~\ref{l.3}(b) for $V$, and the cutoff function
$\rho$ introduced below, none of which depend on $\epsilon$.

Let $\Sigma\subset H_2(X,\mathbb Z)$ be the
image of the map $H_2(V',\mathbb Z)\to H_2(X,\mathbb Z)$ induced by
inclusion. Since $\omega=dd^ch$ is exact on $V'$, we have $\int_S\omega=0$
for all $S\in\Sigma$. Applying Lemma~\ref{l.ra} with this $\Sigma$, with $k=n+2$, with
$\epsilon$, and with the compact set $D$, we obtain a smooth, closed, real
$(1,1)$-form $\tilde\omega_D$ on $X$ such that
\[
  \tilde\omega_D>0\ \text{at every point of }D,\qquad
  \|\tilde\omega_D-\omega\|_{C^{n+2}(D)}<\epsilon ,
\]
and whose class $[\tilde\omega_D]\in H^2(X,\mathbb Q)$ is rational and
satisfies
\[
  \int_S \tilde\omega_D=0\quad\text{for all }S\in\Sigma.
\]

Put $\Theta=\tilde\omega_D-\omega$. Then for every
integral $2$-cycle $S'$ in $V'$, its class in $H_2(X,\mathbb Z)$ lies in
$\Sigma$, so
\[
  \int_{S'}\Theta=\int_{S'}\tilde\omega_D-\int_{S'}\omega=0 .
\]
By Prop.~\ref{p.UCT}(a), applied to $V'$, the form $\Theta$ is
$d$-exact on $V'$, which is a neighbourhood of $\overline V$. Finally, since $V\subset D$,
\[
  \|\Theta\|_{C^{n+2}(V)}\le\|\Theta\|_{C^{n+2}(D)}<\epsilon .
\]

By construction, $\Theta=\tilde\omega_D-\omega$ is a smooth real $(1,1)$-form on
$X$ which is $d$-exact on a neighbourhood of $\overline V$, and
$\|\Theta\|_{C^{n+2}(V)}<\epsilon$. By Lemma~\ref{l.3}(b), applied to the strictly
pseudoconvex domain $V$, there exists $u\in C^\infty(\overline V)$ with
$dd^cu=\Theta$ on $V$ and
\[
  \|u\|_{C^2(V)}\le C\,\|\Theta\|_{C^{n+2}(V)}<C\epsilon,
\]
where $C$ depends only on $V$ and the metric.

Fix an open neighbourhood $W$ of $K$ with $\overline W\subset V$, and a smooth
function $\rho\colon X\to[0,1]$ with $\rho\equiv 1$ on $W$ and
$\operatorname{supp}\rho\Subset V$. The function $\rho u$, extended by zero,
is smooth on $X$. Define
\[
  \tilde\omega=\tilde\omega_D-dd^c(\rho u).
\]
This is a smooth, closed, real $(1,1)$-form on $X$ with
$[\tilde\omega]=[\tilde\omega_D]$. Applying~\eqref{eq:ddc-product} to
$dd^c(\rho u)$ and using $dd^cu=\tilde\omega_D-\omega$ on $V$, we obtain on $V$
\[
  \tilde\omega=(1-\rho)\,\tilde\omega_D+\rho\,\omega-E,\quad\text{where}\quad
  E=u\,dd^c\rho+d\rho\wedge d^cu+du\wedge d^c\rho .
\]
Outside $\operatorname{supp}\rho$ we have $\tilde\omega=\tilde\omega_D$,
and on $W$ we have $E=0$ and $\tilde\omega=\omega$. Moreover, there is
a constant $C_\rho$, depending only on $\|\rho\|_{C^2}$ and the metric, such
that
\[
  \|E\|_{C^0(V)}\le C_\rho\,\|u\|_{C^2(V)}<C_\rho\,C\,\epsilon .
\]

Let $g$ be the fundamental form of the fixed Hermitian metric, and write
$\alpha\ge\beta$ if $\alpha-\beta$ is positive semidefinite. There is a
constant $c_0>0$ such that $\alpha\ge -c_0\|\alpha\|_{C^0(Q)}\,g$ on $Q$ for
every real $(1,1)$-form $\alpha$ and every set $Q$.

Since $\omega$ is positive definite and $D$ is
compact, there is $m>0$ with $\omega\ge 2m\,g$ on $D$. We now fix
$\epsilon>0$ so small that
\[
  c_0\,\epsilon\le m\qquad\text{and}\qquad c_0\,C_\rho\,C\,\epsilon<m .
\]
This choice depends only on $D$, $V$ (through $C$) and $\rho$ (through
$C_\rho$).

\emph{Positivity on $D$.} Since $\|\tilde\omega_D-\omega\|_{C^0(D)}<\epsilon$,
we have $\tilde\omega_D\ge\omega-c_0\epsilon\,g\ge m\,g$ on $D$. On $V$, the
form $(1-\rho)\tilde\omega_D+\rho\,\omega$ is a convex combination of forms
$\ge m\,g$, hence
\[
  \tilde\omega \ge m\,g-c_0\|E\|_{C^0(V)}\,g
  >(m-c_0C_\rho C\epsilon)\,g\ge 0 .
\]
On $D\setminus\operatorname{supp}\rho$ we have
$\tilde \omega =\tilde\omega_D\ge m\,g$. It follows that $\tilde \omega$
is positive definite on $D=\{\Psi\le R\}$.

\emph{Global positivity.} The set $D=\{\Psi\le R\}$ is compact and
$\mathcal O(X)$-convex, because it is convex with respect to the
plurisubharmonic function $\Psi$, and on a Stein manifold the
plurisubharmonic and holomorphic hulls of a compact set coincide
\cite[Sec.~5.1]{Hormander}. Since $\tilde\omega$ is positive definite on
$D$, Lemma~\ref{l.1}(b) gives a K\"ahler form $\Omega$ on $X$ with
$[\Omega]=[\tilde\omega]=[\tilde\omega_D]$ and $\Omega=\tilde\omega$ on $D$.
Since $W\subset V\subset D$, we have $\Omega=\tilde\omega=\omega=dd^ch$ on $W$.

\emph{Integrality.} Since $[\Omega]=[\tilde\omega_D]$ is rational and
$H_2(X,\Z)$ is finitely generated, by Corollary~\ref{c.F} there is an integer
$N\ge1$ such that $N[\Omega]$ lies in the Hodge lattice $\mathcal H_X$.
Thus $\Omega_H:=N\Omega$ is a
Hodge form on $X$, and on the neighbourhood $W$ of $K$,
\[
  \Omega_H=N\,dd^ch=dd^c(Nh).
\]
By~\cite{BGS}, $K$ is meromorphically convex in $X$.

\subsection{Case 2: the general case}

We now only assume~\eqref{e.h2}: $H_2(X,\Z)=\mathcal F\oplus\mathcal T$ with
$\mathcal F$ free, possibly of infinite rank, and $\mathcal T$ torsion. We use the notation
of Sec.~\ref{ssec:decomp}; in particular $\operatorname{pr}\colon H_2(X,\Z)\to\mathcal F$ is the projection
with kernel $\mathcal T$.

Let $n$, the Hermitian metric, $V\Subset V'\Subset U$, $\Psi$, $R$,
$D=\{\Psi\le R\}$, $W$ and $\rho$ be as in Case~1. Let $\epsilon>0$ be the constant fixed at the end of Case~1; 
it depends only on $V$, $D$, the constant of
Lemma~\ref{l.3}(b) for $V$, and $\rho$.

Choose a regular value $R_1>R$ of $\Psi$ and set $X_1=\{\Psi<R_1\}$. Then
$X_1$ is Stein,  $\overline{X_1}$ is a compact manifold with
boundary, and so $H_2(X_1,\Z)$ is finitely generated. Let
\[
  \Phi_1\colon H_2(X_1,\Z)\to H_2(X,\Z),\qquad \mathcal K_1=\ker\Phi_1,
\]
where $\Phi_1$ is induced by inclusion.

\medskip

Let $\Sigma\subset H_2(X_1,\Z)$ be the union of $\mathcal K_1$ and the image
of $H_2(V',\Z)\to H_2(X_1,\Z)$. The K\"ahler form $\omega|_{X_1}$ has zero
periods on $\Sigma$: on the image of $H_2(V',\Z)$ because $\omega=dd^ch$ is
exact on $V'$, and on $\mathcal K_1$ because
$\int_S\omega=\langle[\omega],\Phi_1(S)\rangle=0$ for $S\in\mathcal K_1$.

Apply Lemma~\ref{l.ra} to the Stein manifold $X_1$, the K\"ahler form
$\omega|_{X_1}$, this $\Sigma$, $k=n+2$, $\epsilon$, and the compact set
$D\subset X_1$. We obtain a smooth, closed, real $(1,1)$-form
$\tilde\omega_D$ on $X_1$, positive definite on $D$, with
$\|\tilde\omega_D-\omega\|_{C^{n+2}(D)}<\epsilon$, whose class
$[\eta]:=[\tilde\omega_D]\in H^2(X_1,\Q)$ is rational and has zero periods on
$\Sigma$.

The argument in Case~1 now
apply verbatim, with $X$ replaced by $X_1$: they only use that $\tilde\omega_D$
is a closed real $(1,1)$-form on a neighbourhood of $D$ with these
properties, whose class vanishes on the image of $H_2(V',\Z)$. Therefore, we obtain, as
in Case~1, a smooth, closed, real $(1,1)$-form $\tilde\omega$ on
$X_1$ such that
\[
  [\tilde\omega]=[\eta]\in H^2(X_1,\R),\qquad
  \tilde\omega>0\ \text{on }D,\qquad
  \tilde\omega=\omega\ \text{on }W .
\]

\medskip

Let $I_1=\operatorname{pr}\bigl(\Phi_1(H_2(X_1,\Z))\bigr)\subset\mathcal F$, a finitely
generated subgroup, and let $\mathcal S_1$ be its saturation in $\mathcal F$.
By Lemma~\ref{l.sat}, $\mathcal S_1$ is a finitely generated direct summand of
$\mathcal F$; write $\mathcal F=\mathcal S_1\oplus\mathcal B$.

Define
$\phi\colon\Phi_1(H_2(X_1,\Z))\to\Q$ by $\phi(\Phi_1(z))=\langle[\eta],z\rangle$.
This is well-defined because $[\eta]$ vanishes on $\mathcal K_1$, and it is
$\Q$-valued because $[\eta]$ is rational. If $\Phi_1(z)\in\mathcal T$, then
$n\Phi_1(z)=0$ for some $n\ge1$, so $nz\in\mathcal K_1$ and
$\langle[\eta],z\rangle=0$. Hence, $\phi$ vanishes on
$\Phi_1(H_2(X_1,\Z))\cap\mathcal T$, and so $\phi=\bar\phi\circ\operatorname{pr}$ for a
homomorphism $\bar\phi\colon I_1\to\Q$. Since $\mathcal S_1/I_1$ is torsion and
$\Q$ is uniquely divisible, $\bar\phi$ extends uniquely to a homomorphism
$\psi\colon\mathcal S_1\to\Q$, namely $\psi(x)=\bar\phi(nx)/n$ whenever
$nx\in I_1$. Let $C_1,\dots,C_m$ be a basis of $\mathcal S_1$, and let $N\ge1$
be a common denominator of $\psi(C_1),\dots,\psi(C_m)$; then
$N\psi(\mathcal S_1)\subset\Z$.

Let $c_\psi\in H^2(X,\R)$ be the class given by Lemma~\ref{l.pre}, with
$\mathcal S=\mathcal S_1$. It has the following two properties.

(i) \emph{$c_\psi|_{X_1}=[\eta]$.} For $z\in H_2(X_1,\Z)$ we have
$\operatorname{pr}(\Phi_1(z))\in I_1\subset\mathcal S_1$, so
\[
  \langle c_\psi|_{X_1},z\rangle=\chi_\psi(\Phi_1(z))=\psi\bigl(\operatorname{pr}(\Phi_1(z))\bigr)
  =\bar\phi\bigl(\operatorname{pr}(\Phi_1(z))\bigr)=\phi(\Phi_1(z))=\langle[\eta],z\rangle .
\]
By Prop.~\ref{p.UCT}(a) on $X_1$, the two classes coincide in
$H^2(X_1,\R)$.

(ii) \emph{$Nc_\psi\in\mathcal H_X$.} This is the last assertion of Lemma~\ref{l.pre}.

\medskip

By Lemma~\ref{l.1}(a), we may represent $c_\psi$
by a smooth, closed, real $(1,1)$-form $\Omega_\psi$ on $X$. By (i),
$[\Omega_\psi]$ and $[\tilde\omega]$ coincide in $H^2(X_1,\R)$, so
$\tilde\omega-\Omega_\psi$ is exact on the Stein manifold $X_1$. By
Lemma~\ref{l.3}(a), there is a real function $f\in C^\infty(X_1)$ with
$\tilde\omega=\Omega_\psi+dd^cf$ on $X_1$. Choose a smooth function
$\tilde\rho\colon X\to[0,1]$ with $\tilde\rho\equiv1$ on a neighbourhood of
$D$ and $\operatorname{supp}\tilde\rho\Subset X_1$, and put
\[
  \hat\omega=\Omega_\psi+dd^c(\tilde\rho f).
\]
This is a smooth, closed, real $(1,1)$-form on $X$ with
$[\hat\omega]=c_\psi$, and $\hat\omega=\tilde\omega$ on $D$. In
particular $\hat\omega$ is positive definite on $D$, and $\hat\omega=\omega$
on $W$.

The set $D=\{\Psi\le R\}$ is compact and $\mathcal O(X)$-convex (see 
Case~1), and $\hat\omega$ is positive definite on $D$. By
Lemma~\ref{l.1}(b) there is a K\"ahler form $\Omega$ on $X$ with
$[\Omega]=[\hat\omega]=c_\psi$ and $\Omega=\hat\omega$ on $D$. Since
$W\subset V\subset D$ and $\hat\omega=\omega=dd^ch$ on $W$, we have
$\Omega=dd^ch$ on $W$.

By (ii), $\Omega_H:=N\Omega$ is a Hodge form on $X$, and on the
neighbourhood $W$ of $K$,
\[
  \Omega_H=N\,dd^ch=dd^c(Nh).
\]
By~\cite{BGS}, $K$ is meromorphically convex in $X$. This completes the proof of
Theorem~\ref{t.main}.

\subsection{Trivial extension}\label{s.triv}
In this subsection we prove a proposition showing that if the topology of $X$ near the hull of $K$
is trivial in a suitable sense, then no assumption on $H_2(X,\Z)$ is needed, and
a K\"ahler extension of $dd^ch$ can even be taken exact. Recall that the
holomorphically convex hull of $K$ is
$\widehat K_{\mathcal O(X)}=\{x\in X:\ |f(x)|\le\max_K|f|\ \text{for all } f\in\mathcal O(X)\}$.

\begin{prop}\label{prop:exact}
Let $X$ be a Stein manifold, and let $U$, $h$ and $K=\{h\le0\}$ be as in
Theorem~\ref{t.main}, but without assuming~\eqref{e.h2}. Suppose that $dd^ch$
extends to a K\"ahler form $\omega$ on $X$, and that $\widehat K_{\mathcal O(X)}$ has a
neighbourhood $W$ such that the image of $H_2(W,\Z)\to H_2(X,\Z)$ consists of
torsion classes. Then there is a smooth strictly plurisubharmonic function
$\varphi$ on $X$ with $dd^c\varphi=dd^ch$ on a neighbourhood of $K$.
In particular, by~\cite{BGS}, $K$ is convex with respect to principal hypersurfaces.
\end{prop}

\begin{proof}
By \cite[Theorem~5.1.6]{Hormander} there is a smooth strictly
plurisubharmonic exhaustion function $\Psi$ on $X$ with $\Psi<0$ on
$\widehat K_{\mathcal O(X)}$ and $\Psi>0$ on $X\setminus W$. Fix $R_0$ and $R_0'$ with
$\max_{\widehat K_{\mathcal O(X)}}\Psi<R_0<R_0'<0$, and put
\[
  X_0=\{\Psi\le R_0\},\qquad U_0=\{\Psi<R_0'\}.
\]
Then $X_0$ is compact, $K\subset\widehat K_{\mathcal O(X)}\subset\operatorname{int}X_0$, and
$U_0\subset W$ is Stein. The map
$H_2(U_0,\Z)\to H_2(X,\Z)$ factors through $H_2(W,\Z)$, so its image
consists of torsion classes, on which all periods vanish.

Hence, the homomorphism $\langle[\omega],\cdot\rangle\colon H_2(X,\Z)\to\R$
vanishes on the image of $H_2(U_0,\Z)$. By Lemma~\ref{l:loc}, applied to
$U_0$ and the compact set $X_0$, there is a smooth, closed, real
$(1,1)$-form $\beta$ on $X$ with periods $\langle[\omega],\cdot\rangle$
which vanishes on a neighbourhood of $X_0$. Put $\hat\omega=\omega-\beta$.
All periods of $\hat\omega$ vanish, so $\hat\omega$ is exact by
Prop.~\ref{p.UCT}(a), and $\hat\omega=\omega>0$ on $X_0$.

The set $X_0$ is compact and $\mathcal O(X)$-convex, being a sublevel set of
the strictly plurisubharmonic exhaustion function $\Psi$ (see Case~1). 
By Lemma~\ref{l.1}(b) there is a K\"ahler form $\Omega$ on $X$ with
$[\Omega]=[\hat\omega]=0$ and $\Omega=\hat\omega=\omega$ on $X_0$. Thus
$\Omega$ is exact, and $\Omega=dd^ch$ on a neighbourhood of $K\subset\operatorname{int}X_0$.
By Lemma~\ref{l.3}(a) there is $\varphi\in C^\infty(X)$ with
$dd^c\varphi=\Omega$, and $\varphi$ is strictly plurisubharmonic because
$\Omega$ is a K\"ahler form.
\end{proof}

\section{Proof of Theorem~\ref{thm:counterexample}}\label{sec:counterexample}

We first outline the idea; $t=\log|w|$ denotes the height in the fibre of
$X=Y\times\mathbb{C}^*$. The torus $\Lambda$ lies over a figure-eight
$\Gamma\subset Y$ with double point at $0\in\mathbb{R}^5$. Over one loop of
$\Gamma$ the height rises by $1$, and over the other loop it returns, so
$\Lambda$ meets the fibre over $0$ in two circles $C_0,C_1$ at heights $0$ and
$1$. They cobound the holomorphic annulus $A=\{0\}\times\{1\le|w|\le e\}$ and
an annulus $B\subset\Lambda$. A Kähler form $\Omega$ for which $\Lambda$ is
isotropic integrates to zero over $B$ and to a positive number over $A$, so it
has a positive period on the cycle $A-B$; this is what rules out Hodge
extensions, and, with positivity of intersections, the meromorphic convexity of
$\Lambda$. The difficulty is to find $h$ vanishing to second order exactly on
$\Lambda$ with $dd^ch$ equal to a \emph{global} Kähler form. The natural
potential degenerates along $C_0\cup C_1$. Since $\Gamma$ is straight near the
double point, this degeneracy is explicit, and we remove it by adding small
multiples of explicit pluriharmonic functions, cut off away from
$C_0\cup C_1$ (Lemma~\ref{lem:correction}).

\subsection{The manifold}\label{ssec:manifold}

\begin{lemma}\label{lem:Y}
There exist an open set $U_M\subset\mathbb{R}^5$ and a Stein domain
$Y\subset\mathbb{C}^5$ with $Y\cap\mathbb{R}^5=U_M$ such that $U_M$ is a
deformation retract of $Y$ and
\[
H_1(Y,\mathbb{Z})\cong\mathbb{Q},\qquad H_2(Y,\mathbb{Z})=0 .
\]
\end{lemma}

\begin{proof}
Recall that the \emph{mapping cylinder} of a continuous map
$f\colon A\to B$ is the space
\[
  M_f=\bigl((A\times[0,1])\sqcup B\bigr)/\bigl((a,1)\sim f(a)\bigr),
\]
obtained by attaching the end $A\times\{1\}$ of the cylinder $A\times[0,1]$
to $B$ via $f$. It contains $A\cong A\times\{0\}$ and $B$, and it deformation
retracts onto $B$ by sliding each segment $\{a\}\times[0,1]$ to its endpoint
$f(a)$. Given a sequence of maps $A_1\xrightarrow{f_1}A_2\xrightarrow{f_2}\cdots$,
its \emph{mapping telescope} is the union of the mapping cylinders $M_{f_k}$,
where the target end $A_{k+1}$ of $M_{f_k}$ is identified with the source end
$A_{k+1}\times\{0\}$ of $M_{f_{k+1}}$; see \cite[Ch.~0 and Sec.~3.F]{Hatcher}.

Let $S_1,S_2,\dots$ be copies of the circle, let $f_k\colon S_k\to S_{k+1}$
be the map $z\mapsto z^{k+1}$, and let $M_k$ be its mapping cylinder. Let
$M$ be the mapping telescope of the sequence
$$
S_1\xrightarrow{f_1}S_2\xrightarrow{f_2}S_3\xrightarrow{f_3}\cdots .
$$
This is obtained as the union of the $M_k$ glued along $M_k\cap M_{k+1}=S_{k+1}$, i.e.,
by gluing the target $S^1$ of $f_k$ with the source $S^1$ of $f_{k+1}$.

\emph{Homology of $M$.} Put $T_0=S_1$ and $T_k=M_1\cup\dots\cup M_k$
for $k\ge1$. Since each mapping cylinder retracts onto its target, the
compact set $T_k$ deformation retracts onto $S_{k+1}$. So
$H_1(T_k,\mathbb{Z})\cong\mathbb{Z}$, generated by $[S_{k+1}]$, and
$H_2(T_k,\mathbb{Z})=0$. Under the inclusion $T_k\subset T_{k+1}$ followed
by the retraction onto $S_{k+2}$, the circle $S_{k+1}$ maps by $f_{k+1}$,
which has degree $k+2$. Therefore, the induced map $H_1(S_{k+1}, \Z) \to H_1(S_{k+2}, \Z)$
is multiplication by $k+2$. 

Every compact subset of $M$ lies in some $T_k$, so
by \cite[Prop.~3.33]{Hatcher}
\[
H_1(M,\mathbb{Z})\cong\varinjlim\bigl(\mathbb{Z}\xrightarrow{\,\cdot 2\,}\mathbb{Z}\xrightarrow{\,\cdot 3\,}\mathbb{Z}\xrightarrow{\,\cdot 4\,}\cdots\bigr)\cong\mathbb{Q},
\qquad H_2(M,\mathbb{Z})=0 ,
\]
To evaluate this direct limit, map the $k$-th copy of $\mathbb{Z}$ injectively into $\mathbb{Q}$ by sending its generator $1$ to $\frac{1}{k!}$. This mapping is well-defined and commutes with the directed system because the transition from the $k$-th stage to the $(k+1)$-th stage is multiplication by $k+1$, and $(k+1) \cdot \frac{1}{(k+1)!} = \frac{1}{k!}$.
The limit maps onto all of $\mathbb{Q}$ because any rational number $\frac{p}{q}$ can be written as $\frac{p(q-1)!}{q!}$, which lies in the image of the $q$-th stage. Therefore, the direct limit is the additive group of rational numbers.

\emph{Embedding $M$ in $\mathbb{R}^5$.} Triangulate $M$ inductively, so that
the triangulations of $M_k$ and $M_{k+1}$ agree on $S_{k+1}$ (the maps $f_k$
cannot all be simplicial for a fixed triangulation of the circles, but the
cylinder $M_{k+1}$ can be triangulated with the already fixed triangulation of
$S_{k+1}$ at its source end, after subdividing its target end). This makes $M$
a countable, locally finite simplicial complex of dimension $2$. Enumerate its
vertices as $v_1,v_2,\dots$ and put
$p_j=(j,j^2,j^3,j^4,j^5)\in\mathbb{R}^5$. Any six distinct points $p_j$ are
affinely independent, since the matrix with rows $(1,t,\dots,t^5)$ is a
Vandermonde matrix. 
Let $\Phi\colon M\to\mathbb{R}^5$ be the map that sends $v_j$ to $p_j$ for
every $j$ and is affine on each simplex of $M$, so that a simplex with
vertices $v_{j_0},\dots,v_{j_k}$ is mapped onto the simplex spanned by
$p_{j_0},\dots,p_{j_k}$. Two simplices of $M$ have at most $6$ vertices
altogether, so by uniqueness of barycentric coordinates $\Phi$ is injective:
two simplices meet exactly in their common face, and so do their
images. It is proper, because only finitely many simplices have a vertex
$v_j$ with $j\le J$, and all the other simplices lie in $\{x_1>J\}$. Thus, $M$
is homeomorphic to a closed locally finite polyhedron in $\mathbb{R}^5$, which
we again denote by $M$.

\emph{Thickening.} Since $M$ is a closed subpolyhedron of $\mathbb{R}^5$,
it has an open neighbourhood $U_M\subset\mathbb{R}^5$ of which it is a
deformation retract, for instance the interior of a regular neighbourhood
\cite[Ch.~3]{RS}. By \cite[Lemma~6]{Na}, there exists a Stein domain $Y\subset \mathbb{C}^5$
such that $Y\cap \R^5 = U_M$ and $U_M$ is a deformation retract of $Y$. 
Then $Y$ is homotopy equivalent to $M$, which proves the lemma.
\end{proof}

\begin{remark}
For $k=1$ and $G=\mathbb{Q}$, \cite[Theorem~5]{Na} produces a Runge domain
$D\subset\mathbb{C}^5$ with $H_1(D,\mathbb{Z})\cong\mathbb{Q}$, but gives no
control of $H_2(D,\mathbb{Z})$. The telescope $M$ is a concrete complex for
which $H_2(M,\mathbb{Z})=0$ can be verified directly, as needed in
Lemma~\ref{lem:Y}.
\end{remark}

From now on $Y$ and $U_M$ are as in Lemma~\ref{lem:Y}. A translation
$z\mapsto z-q$ of $\mathbb{C}^5$ by a real vector $q\in\mathbb{R}^5$ is a
biholomorphism that maps $\mathbb{R}^5$ onto itself, so it preserves all the
properties in Lemma~\ref{lem:Y}. Translating $Y$ and $U_M$ by a point of $U_M$,
we assume from now on that $0\in U_M$. We denote by $e_1,\dots,e_5$ the standard
basis of $\mathbb{R}^5$, by $\mathbb{R}_{x_j}=\{x\in\mathbb{R}^5:\ x_k=0 \text{ for } k\ne j\}$
the real $x_j$-axis, and we write $z=x+iy$, with $x=\operatorname{Re}z$ and
$y=\operatorname{Im}z$ in $\mathbb{R}^5$, for the coordinates on $\mathbb{C}^5$.
For $r>0$ let $\mathbb{B}_r=\{z\in\mathbb{C}^5:|z|<r\}$; its trace
$\mathbb{B}_r\cap\mathbb{R}^5$ is the ball of radius $r$ in $\mathbb{R}^5$. Set
$X=Y\times\mathbb{C}^*$, with coordinates $(z,w)$, projection
$\pi\colon X\to Y$ and $t=\log|w|$.

\begin{lemma}\label{lem:topX}
$X$ is a Stein manifold of dimension $6$, and
\begin{equation}\label{eq:topX}
H_2(X,\mathbb{Z})\cong\mathbb{Q},\quad
\operatorname{Hom}(H_2(X,\mathbb{Z}),\mathbb{Z})=0,\quad
H^2(X,\mathbb{R})\cong\mathbb{R},\quad
\operatorname{Im}\bigl(H^2(X,\mathbb{Z})\to H^2(X,\mathbb{R})\bigr)=0 .
\end{equation}
Moreover, $H^2(X,\mathbb{Z})\cong\operatorname{Ext}(\mathbb{Q},\mathbb{Z})\neq0$.
\end{lemma}

\begin{proof}
$X$ is Stein as a product of Stein manifolds. 
Since $H_*(\mathbb{C}^*,\mathbb{Z})\cong H_*(S^1,\mathbb{Z})$ is free, the
K\"unneth formula has no $\operatorname{Tor}$ terms and gives
\begin{align*}
H_1(X,\mathbb{Z})&\cong\bigl(H_1(Y)\otimes H_0(S^1)\bigr)\oplus\bigl(H_0(Y)\otimes H_1(S^1)\bigr)
\cong(\mathbb{Q}\otimes\mathbb{Z})\oplus(\mathbb{Z}\otimes\mathbb{Z})\cong\mathbb{Q}\oplus\mathbb{Z},\\
H_2(X,\mathbb{Z})&\cong\bigl(H_2(Y)\otimes H_0(S^1)\bigr)\oplus\bigl(H_1(Y)\otimes H_1(S^1)\bigr)\oplus\bigl(H_0(Y)\otimes H_2(S^1)\bigr)
\cong0\oplus(\mathbb{Q}\otimes\mathbb{Z})\oplus0\cong\mathbb{Q},
\end{align*}
where all homology groups are taken with coefficients in $\mathbb{Z}$.

The remaining statements follow from Prop.~\ref{p.UCT}. Indeed,
$\operatorname{Hom}(\mathbb{Q},\mathbb{Z})=0$: for a homomorphism $f$ and
$x\in\mathbb{Q}$, the integer $f(x)$ is divisible by every $n\ge1$, so
$f(x)=0$ (whereas $\operatorname{Hom}(\mathbb{Q},\mathbb{R})\cong\mathbb{R}$).
Hence, $H^2(X,\mathbb{R})\cong\mathbb{R}$ by Prop.~\ref{p.UCT}(a),
$\mathcal H_X=0$ by Prop.~\ref{p.UCT}(c), and
\[
H^2(X,\mathbb{Z})\cong\operatorname{Ext}(H_1(X,\mathbb{Z}),\mathbb{Z})
=\operatorname{Ext}(\mathbb{Q}\oplus\mathbb{Z},\mathbb{Z})=\operatorname{Ext}(\mathbb{Q},\mathbb{Z}),
\]
which is nonzero (it is a $\mathbb{Q}$-vector space of infinite dimension, see
\cite[Example~1]{Co}).
\end{proof}

So $X$ carries nontrivial holomorphic line bundles, since
$H^1(X,\mathcal{O}^*)\cong H^2(X,\mathbb{Z})\neq0$, but none of them is
detected by real cohomology.
Also, a Kähler form on $X$ is a Hodge form if and only if it is exact.

Let $\ell\colon H_1(Y,\mathbb{Z})\cong\mathbb{Q}\hookrightarrow\mathbb{R}$
be the inclusion, and let $a\in H_1(Y,\mathbb{Z})$ be the class with
$\ell(a)=1$. Since $H_0(Y,\mathbb{Z})$ is free, the $\operatorname{Ext}$ term
in~\eqref{eq:UCT} vanishes for $k=1$, so
$H^1(Y,\mathbb{C})\to\operatorname{Hom}(H_1(Y,\mathbb{Z}),\mathbb{C})$ is an
isomorphism; composed with the de Rham isomorphism, it is the period map. Hence,
the homomorphism $\ell$, regarded as $\mathbb{C}$-valued, is the period
homomorphism of a class in $H^1(Y,\mathbb{C})$.
Since $Y$ is Stein, every class in $H^1(Y,\mathbb{C})$ is represented by a closed
holomorphic $1$-form \cite{Serre}. Thus, there is a closed holomorphic $1$-form
$\beta$ on $Y$ whose periods are given by $\ell$:
\[
  \int_c\beta=\ell(c)\qquad\text{for every } c\in H_1(Y,\mathbb{Z}).
\]
Put $\alpha=\operatorname{Re}\beta$. Since $\ell$ is real-valued,
$\int_c\alpha=\ell(c)$ as well. Locally
$\beta=dG$ with $G$ holomorphic, and then $\alpha=d\operatorname{Re}G$. On
$X$ consider the closed holomorphic $1$-form
\begin{equation}\label{e.mu}
\mu=\tfrac12\Bigl(\frac{dw}{w}-\pi^*\beta\Bigr).
\end{equation}
Locally $\mu=\partial v$, where $v=\log|w|-\operatorname{Re}G\circ\pi$ is
pluriharmonic. Globally $2\operatorname{Re}\mu=dt-\pi^*\alpha$.
Summarizing, the data from this subsection that will be used below are the
injective homomorphism $\ell\colon H_1(Y,\mathbb{Z})\to\mathbb{R}$, the class
$a\in H_1(Y,\mathbb{Z})$, and the closed real $1$-form $\alpha$ on $Y$, related by
\begin{equation}\label{eq:ell-alpha}
  \ell(a)=1,\qquad \int_c\alpha=\ell(c)\quad\text{for all } c\in H_1(Y,\mathbb{Z}).
\end{equation}

\subsection{The curve and the torus}\label{ssec:curve}

\begin{lemma}\label{lem:curve}
There exist a smooth immersion
$\gamma\colon\mathbb{R}/2\mathbb{Z}\to U_M\subset Y$ such that
\begin{enumerate}
\item[(a)] $\gamma(s)=\gamma(s')$ with $s\ne s'$ if and only if
  $\{s,s'\}=\{0,1\}$, with $\gamma(0)=\gamma(1)=0$;
\item[(b)] the loops $\gamma|_{[0,1]}$ and $\gamma|_{[1,2]}$ represent $a$
  and $-a$ in $H_1(Y,\mathbb{Z})$;
\item[(c)] there is $r_\Gamma\in(0,\tfrac12)$ such that $\gamma(s)=se_1$ for
  $|s|<r_\Gamma$, $\gamma(s)=(s-1)e_2$ for $|s-1|<r_\Gamma$, and
  $|\gamma(s)|\ge r_\Gamma$ for all other $s\in\mathbb{R}/2\mathbb{Z}$.
\end{enumerate}
In particular, for every $r\in(0,r_\Gamma]$ the curve $\Gamma=\gamma(\mathbb{R}/2\mathbb{Z})$
satisfies $\Gamma\cap\mathbb{B}_r=(\mathbb{R}_{x_1}\cup\mathbb{R}_{x_2})\cap\mathbb{B}_r$.
\end{lemma}

\begin{proof}
Fix $r_\Gamma\in(0,\tfrac12)$ with $\{x\in\mathbb R^5:|x|\le r_\Gamma\}\subset U_M$.
Since $M$ is connected, so is $U_M$, and $U_M\hookrightarrow Y$ is a homotopy
equivalence; by the Hurewicz theorem there are loops in $U_M$ based at $0$
representing $a$ and $-a$. After a homotopy rel $0$ inside $U_M$, a smooth
approximation, and general position in $\mathbb R^5$ we obtain smooth regular loops 
$b_1,b_2\colon[0,1]\to U_M$ based at $0$, representing $a$ and $-a$, such that
\begin{enumerate}
\item[(1)] $b_1(s)=se_1$, $b_1(1-s)=-se_2$, $b_2(s)=se_2$,
  $b_2(1-s)=-se_1$ for $0\le s<r_\Gamma$, and $|b_j(s)|\ge r_\Gamma$ for
  $s\in[r_\Gamma,1-r_\Gamma]$;
\item[(2)] each $b_j$ is injective on $[0,1)$, and $b_1(s)\ne b_2(s')$
  unless $s,s'\in\{0,1\}$.
\end{enumerate}
Thus $b_1$ leaves $0$ along $+e_1$ and returns along the direction $e_2$ from the
side of $-e_2$, while $b_2$ leaves along $+e_2$ and returns along $e_1$ from
the side of $-e_1$. Define $\gamma(s)=b_1(s)$ for $s\in[0,1]$ and
$\gamma(s)=b_2(s-1)$ for $s\in[1,2]$. By (1),
$\gamma(s)=se_1$ for $|s|<r_\Gamma$ in $\mathbb R/2\mathbb Z$,
$\gamma(s)=(s-1)e_2$ for $|s-1|<r_\Gamma$, and $|\gamma(s)|\ge r_\Gamma$
otherwise, so $\gamma$ is a smooth immersion and (c) holds. Property (a)
follows from (1) and (2), and (b) holds by construction. Finally, if
$0<r\le r_\Gamma$, the points $\gamma(s)$ with $|s|<r_\Gamma$ or $|s-1|<r_\Gamma$
lie on $\mathbb{R}_{x_1}\cup\mathbb{R}_{x_2}$, and all other points of
$\Gamma$ lie outside $\mathbb{B}_r$; this gives the last assertion.
\end{proof}

By Lemma~\ref{lem:curve}(b) and \eqref{eq:ell-alpha},
$\int_\gamma\alpha=\int_{\mathbb{R}/2\mathbb{Z}}\gamma^*\alpha=\ell(a)+\ell(-a)=0$.
Hence, there is a smooth function
$\tau\colon\mathbb{R}/2\mathbb{Z}\to\mathbb{R}$ with
$d\tau=\gamma^*\alpha$ and $\tau(0)=0$. Then, again by
Lemma~\ref{lem:curve}(b) and \eqref{eq:ell-alpha},
$\tau(1)=\int_{\gamma|_{[0,1]}}\alpha=\ell(a)=1$.
Define
\[
\Lambda=\bigl\{\bigl(\gamma(s),e^{\tau(s)+i\theta}\bigr)\in Y \times \C^* :\ s\in\mathbb{R}/2\mathbb{Z},\ \theta\in\mathbb{R}/2\pi\mathbb{Z}\bigr\} .
\]
By Lemma~\ref{lem:curve}(a) and $\tau(1)\neq\tau(0)$, the map
$(s,\theta)\mapsto(\gamma(s),e^{\tau(s)+i\theta})$ is an injective
immersion. Thus, $\Lambda \subset X$ is an embedded smooth torus, and it is easy to see that it is totally
real. Further,
\[
\Lambda\cap\pi^{-1}(0)=C_0\cup C_1,\qquad C_0=\{0\}\times\{|w|=1\},\quad C_1=\{0\}\times\{|w|=e\}.
\]

\begin{figure}[ht]
  \centering
  \includegraphics[width=\textwidth]{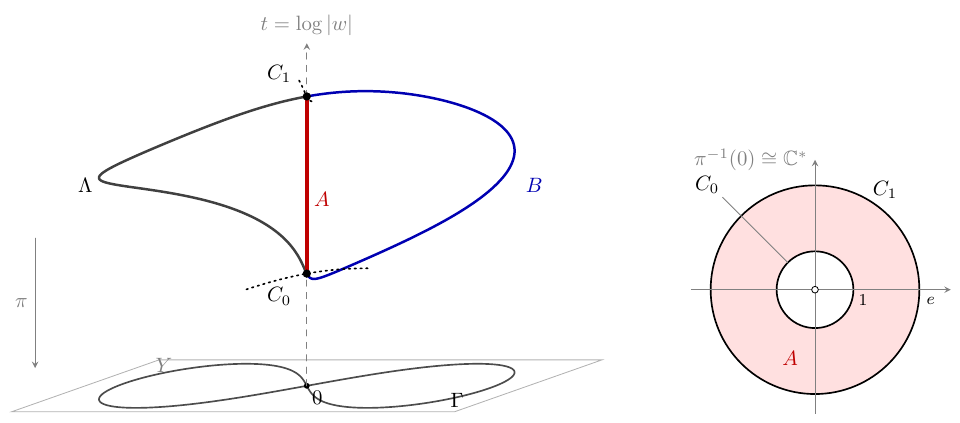}
  \caption{Schematic picture of $\Lambda\subset X$, with $\arg w$ suppressed.
  Left: $\Lambda$ over the figure-eight $\Gamma$; dotted: the extra zero set of
  $h_0$. Right: the fibre $\pi^{-1}(0)$, where $C_0,C_1$ bound the annulus $A$.}
  \label{fig:lambda}
\end{figure}

\subsection{A global Kähler form and a first local potential}\label{ssec:potential}

Let $r_\Gamma$ be as in Lemma~\ref{lem:curve}, and fix $r_B\in(0,r_\Gamma]$
with $\mathbb{B}_{r_B}\subset Y$. By Lemma~\ref{lem:curve},
\begin{equation}\label{eq:Gamma-near-0}
\Gamma\cap\mathbb{B}_{r_B}=(\mathbb{R}_{x_1}\cup\mathbb{R}_{x_2})\cap\mathbb{B}_{r_B}.
\end{equation}
Since $\mathbb{B}_{r_B}$ is simply
connected, there is a holomorphic function $G$ on $\mathbb{B}_{r_B}$ with $dG=\beta$ and
$G(0)=0$. Whenever $\gamma(s)\in\mathbb{B}_{r_B}$, that is, for $|s|<r_B$ or
$|s-1|<r_B$ by Lemma~\ref{lem:curve}(c), we have
$d(\operatorname{Re}G\circ\gamma)=\gamma^*\alpha=d\tau$. Since
$\gamma(0)=\gamma(1)=0$, $G(0)=\tau(0)=0$ and $\tau(1)=1$, it follows that
\begin{equation}\label{eq:G-tau}
\operatorname{Re}G(\gamma(s))=\tau(s)\ \text{ for } |s|<r_B,
\qquad
\operatorname{Re}G(\gamma(s))=\tau(s)-1\ \text{ for } |s-1|<r_B .
\end{equation}
For $x=(x_1,\dots,x_5)\in\mathbb{R}^5$ put
\begin{equation}\label{e.F0}
F_0(x)=(x_1x_2)^2+x_3^2+x_4^2+x_5^2 .
\end{equation}
Then $F_0\ge0$ and $F_0^{-1}(0)=\mathbb{R}_{x_1}\cup\mathbb{R}_{x_2}$, so by \eqref{eq:Gamma-near-0}
$F_0^{-1}(0)\cap\mathbb{B}_{r_B}=\Gamma\cap\mathbb{B}_{r_B}$. At a point
$(x_1,0,0,0,0)$ with $x_1\ne0$ the Hessian of $F_0$ is
$\operatorname{diag}(0,2x_1^2,2,2,2)$; it is positive semidefinite with
kernel $\mathbb{R}_{x_1}=T\Gamma$. The same holds on $\mathbb{R}_{x_2}$ after interchanging
$x_1$ and $x_2$.

\begin{lemma}\label{lem:F}
There are a function $F\in C_c^\infty(\mathbb{R}^5)$ with $F\ge0$, an open
neighbourhood $N$ of $\Gamma$ in $\mathbb{R}^5$ and $r_3>0$ such that
\begin{enumerate}
\item[(a)] $F^{-1}(0)\cap N=\Gamma$;
\item[(b)] $\ker\operatorname{Hess}_zF=T_z\Gamma$ for every $z\in\Gamma\setminus\{0\}$;
\item[(c)] $F=F_0$ on $\mathbb{B}_{r_3}\cap\mathbb{R}^5$.
\end{enumerate}
\end{lemma}

\begin{proof}
We first construct $F$. Fix $0<r_1<r_2<r_B$ and $\chi_0\in C_c^\infty(\mathbb{B}_{r_B}\cap\mathbb{R}^5,[0,1])$
with $\chi_0=1$ on $\mathbb{B}_{r_2}\cap\mathbb{R}^5$. Let $\Gamma'=\Gamma\setminus\mathbb{B}_{r_1}$ and
$\delta=\operatorname{dist}(\cdot,\Gamma')$. Away from $0$ the curve $\Gamma$
is an embedded curve, and $\Gamma\setminus\mathbb{B}_{r_2}$ is a compact subset of
the interior of the embedded curve $\Gamma'$. Hence, $\delta^2$ is smooth on a
neighbourhood $O$ of $\Gamma\setminus\mathbb{B}_{r_2}$, and there it is
nonnegative, it vanishes exactly on $\Gamma'$, and $\operatorname{Hess}_z\delta^2$ is twice the
orthogonal projection onto the normal space of $\Gamma$ at $z$, with kernel
$T_z\Gamma$. On the neighbourhood $N_1=O\cup(\mathbb{B}_{r_2}\cap\mathbb{R}^5)$ of $\Gamma$ put
\[
F_{\rm loc}=\chi_0F_0+(1-\chi_0)\delta^2 .
\]
This is clearly smooth and nonnegative. Further,  $F_{\rm loc}^{-1}(0)=\Gamma$ in $N_1$.

Let $z\in\Gamma\setminus\{0\}$. There are three cases. If $z\in\mathbb{B}_{r_2}$, then
$F_{\rm loc}=F_0$ near $z$, and $\ker\operatorname{Hess}_zF_{\rm loc}=T_z\Gamma$ by the
computation of $\operatorname{Hess}F_0$ above. If $z\notin\mathbb{B}_{r_B}$, then $z\notin\operatorname{supp}\chi_0$,
so $F_{\rm loc}=\delta^2$ near $z$ and the claim holds for $\delta^2$. In the
remaining case $z\in\mathbb{B}_{r_B}\setminus\mathbb{B}_{r_2}$, the functions $F_0$ and
$\delta^2$ are smooth near $z$, nonnegative, and vanish at $z$ (by
\eqref{eq:Gamma-near-0} and $z\in\Gamma'$). Further,  both have zero differential at $z$, and in particular, $g=F_0-\delta^2$ satisfies
$g(z)=0$ and $dg(z)=0$, and the terms involving derivatives of $\chi_0$ in
$\operatorname{Hess}_z(\chi_0 g)$ vanish. Since
$F_{\rm loc}=\delta^2+\chi_0(F_0-\delta^2)$, we get $dF_{\rm loc}(z)=0$ and
\[
\operatorname{Hess}_zF_{\rm loc}=\chi_0(z)\operatorname{Hess}_zF_0+(1-\chi_0(z))\operatorname{Hess}_z\delta^2 .
\]
Both Hessians are positive semidefinite with kernel $T_z\Gamma$, and the coefficients
are nonnegative with sum $1$, so $\ker\operatorname{Hess}_zF_{\rm loc}=T_z\Gamma$.

Finally, choose an open neighbourhood $N\Subset N_1$ of $\Gamma$, and $\theta\in C_c^\infty(N_1,[0,1])$ with
$\theta=1$ on $N$, and put $F=\theta F_{\rm loc}$, extended by zero. Then (a) and
(b) hold, because $F=F_{\rm loc}$ on $N$. Choose $r_3\le r_2$ with
$\mathbb{B}_{r_3}\cap\mathbb{R}^5\subset N$. On $\mathbb{B}_{r_3}\cap\mathbb{R}^5$ we have $\theta=\chi_0=1$,
so $F=F_0$, which shows (c).
\end{proof}

From now on $F$, $N$ and $r_3$ are as in Lemma~\ref{lem:F}, and we fix
$\sigma\ge1$ with $2\sigma>\sup\|\operatorname{Hess}F\|$. Recall that
$\pi\colon X\to Y$ is the projection and $z=x+iy$.

\begin{lemma}\label{lem:rho}
Define
\begin{equation}\label{e.rho}
\rho(z,w)=\sigma\,|\operatorname{Im}z|^2+F(\operatorname{Re}z),\qquad (z,w)\in X=Y\times\mathbb{C}^* ,
\end{equation}
and let $V=\{z\in Y:\operatorname{Re}z\in N\}$, a neighbourhood of $\Gamma$ in $Y$.
Then $\rho$ is a smooth, nonnegative, plurisubharmonic function on $X$, constant
along the fibres of $\pi$, and the following hold.
\begin{enumerate}
\item[(a)] For $q\in X$ and $\xi\in T_qX$ we have $dd^c\rho(\xi,J\xi)>0$ if
  $\pi_*\xi\ne0$, and $dd^c\rho(\xi,\cdot)=0$ if $\pi_*\xi=0$. Thus $\rho$ is not
  strictly plurisubharmonic, and $dd^c\rho$ degenerates exactly in the vertical directions.
\item[(b)] $\{\rho=0\}\cap\pi^{-1}(V)=\pi^{-1}(\Gamma)$, and $d\rho=0$ on $\pi^{-1}(\Gamma)$.
\item[(c)] For $q\in\pi^{-1}(\Gamma\setminus\{0\})$ the kernel of
  $\operatorname{Hess}_q\rho$ is $\{\xi\in T_qX:\ \pi_*\xi\in T_{\pi(q)}\Gamma\}$.
\item[(d)] $\rho(z,w)=\sigma|y|^2+F_0(x)$ for every $(z,w)\in X$ with $z=x+iy$
  and $|x|<r_3$.
\end{enumerate}
\end{lemma}

\begin{proof}
Put $\rho_Y(z)=\sigma\,|\operatorname{Im}z|^2+F(\operatorname{Re}z)$ for
$z\in\mathbb{C}^5$, so that $\rho=\rho_Y\circ\pi$. Clearly $\rho\ge0$ is smooth
and constant along the fibres of $\pi$. By~\eqref{eq:levi-example}, the Levi form
of $\rho_Y$ is at least $\frac14(2\sigma-\sup\|\operatorname{Hess}F\|)|\xi|^2>0$,
so $\rho_Y$ is strictly plurisubharmonic on $\mathbb{C}^5$.

(a) Since $\pi$ is holomorphic, $dd^c\rho=\pi^*dd^c\rho_Y$ and $\pi_*J=J\pi_*$.
Hence,
\[
dd^c\rho(\xi,J\xi)=dd^c\rho_Y(\pi_*\xi,J\pi_*\xi),
\]
which is positive if $\pi_*\xi\neq0$, and $dd^c\rho(\xi,\cdot)=0$ if $\pi_*\xi=0$. In particular $\rho$
is plurisubharmonic.

(b) We have $\rho(z,w)=0$ if and only if $\operatorname{Im}z=0$ and
$F(\operatorname{Re}z)=0$. For $z\in V$ this means $z\in\mathbb{R}^5$ and
$z\in F^{-1}(0)\cap N=\Gamma$, by Lemma~\ref{lem:F}(a). Since $\rho\ge0$
vanishes on $\pi^{-1}(\Gamma)$, $d\rho=0$ there.

(c) Let $q=(z,w)$ with $z\in\Gamma\setminus\{0\}$. In the product coordinates
$(z,w)$ the map $\pi$ is a linear projection, so
$\operatorname{Hess}_q\rho(\xi,\xi)=\operatorname{Hess}_z\rho_Y(\pi_*\xi,\pi_*\xi)$.
Since $z\in\mathbb{R}^5$, the Hessian of $\rho_Y$ at $z$ is the direct sum of
$2\sigma|\cdot|^2$ on $i\mathbb{R}^5$ and $\operatorname{Hess}_zF$ on
$\mathbb{R}^5$. By Lemma~\ref{lem:F}(b) its kernel is $T_z\Gamma$, and (c) follows.

(d) follows from Lemma~\ref{lem:F}(c).
\end{proof}

Define
\[
\widehat\omega=dd^c\rho+i\,\mu\wedge\bar\mu ,
\]
where $\mu$ is defined as in~\eqref{e.mu}.
The form $\widehat\omega$ is closed, because $\mu$ is closed and holomorphic. It is a
Kähler form on $X$ because $i\,\mu\wedge\bar\mu$ is positive on vertical vectors
(where $\mu=\frac{dw}{2w}$), that is, exactly in the directions where
$dd^c\rho$ vanishes by Lemma~\ref{lem:rho}(a).

Let $W\subset X$ be a tubular neighbourhood of $\Lambda$ that retracts onto
$\Lambda$. The closed form $2\operatorname{Re}\mu=dt-\pi^*\alpha$ has zero
period on the fibre circle $\{\theta\mapsto(\gamma(0),e^{i\theta})\}$.
Since $d\tau=\gamma^*\alpha$, it also has zero period on the loop $s\mapsto(\gamma(s),e^{\tau(s)})$,
and so it is exact on $W$. Therefore, there is a pluriharmonic
function $v$ on $W$ with $\partial v=\mu$, $dv=dt-\pi^*\alpha$ and
$v|_\Lambda=0$. Indeed, first observe that $\bar\partial\partial v=\bar\partial\mu=0 $, which gives
pluriharmonicity of $v$;  further, along $\Lambda$, parametrized by $(s,\theta)\mapsto(\gamma(s),e^{\tau(s)+i\theta})$, the form $dv=dt-\pi^*\alpha$ pulls back to $d\tau-\gamma^*\alpha=0$, so $v|_\Lambda$ is constant since $\Lambda$ is connected;
as $v$ is determined only up to an additive constant, we can normalize it so that $v|_\Lambda\equiv0$.

 Put
\[
h_0=\rho+v^2\quad\text{on }W, \qquad dd^c h_0=dd^c\rho+i\,\partial v\wedge\bar\partial v=\widehat\omega ,
\]
where the second identity follows from~\eqref{eq:ddc-square}.
\begin{lemma}\label{lem:h0}
On a neighbourhood of $\Lambda$ the function $h_0$ satisfies the following.
\begin{enumerate}
\item[(a)] $h_0\ge0$, $h_0|_\Lambda=0$ and $dh_0|_\Lambda=0$.
\item[(b)] For $q\in\Lambda\setminus(C_0\cup C_1)$,
  $\operatorname{Hess}_qh_0$ is positive definite on $T_qX/T_q\Lambda$.
\item[(c)] On $\pi^{-1}(\mathbb{B}_{r_B})$ write $z=x+iy$ and consider the
  smooth real function
  \[
  r=t-\operatorname{Re}G(z)=\log|w|-\operatorname{Re}G(z),
  \]
  where $G$ is the function of~\eqref{eq:G-tau}. Near $C_0$ we have
  $\rho=\sigma|y|^2+F_0(x)$ by Lemma~\ref{lem:rho}(d), and
  \[
  \Lambda=\{x_2=x_3=x_4=x_5=0,\ y=0,\ r=0\},\qquad h_0=\rho+r^2 .
  \]
  At a point $q\in\Lambda$ near $C_0$, with $x_1=x_1(q)$,
  \begin{equation}\label{eq:hess-h0}
  \operatorname{Hess}_qh_0=2\sigma\sum_{j=1}^5dy_j^2+2x_1^2\,dx_2^2+2\sum_{j=3}^5dx_j^2+2\,dr^2 ,
  \end{equation}
  where $dy_j^2=dy_j\otimes dy_j$, etc., and all differentials are taken at $q$.
  The ten covectors $dy_1,\dots,dy_5,dx_2,\dots,dx_5,dr$ at $q$ are linearly
  independent and vanish on $T_q\Lambda$. The same holds near $C_1$ after
  interchanging $z_1\leftrightarrow z_2$ and replacing $G$ by $G+1$, that is,
  $r$ by $r-1$.
\end{enumerate}
\end{lemma}

\begin{proof}
(a) holds because both summands of $h_0$ are nonnegative and vanish on
$\Lambda$. For (b), at $q\in\Lambda$ we have
$\operatorname{Hess}_qh_0=\operatorname{Hess}_q\rho+2\,dv\otimes dv$.
Both terms are positive semidefinite, so $\xi$ lies in the kernel of
$\operatorname{Hess}_qh_0$ if and only if it lies in the kernels of both. If
$q\notin C_0\cup C_1$, then $\pi(q)\in\Gamma\setminus\{0\}$, and by
Lemma~\ref{lem:rho}(c) this means that $\pi_*\xi\in T_{\pi(q)}\Gamma$ and $dv(\xi)=0$. 
The vectors $\xi$ with $\pi_*\xi\in T_{\pi(q)}\Gamma$ form a $3$-dimensional
space $E$ containing the vertical space $\ker\pi_*$. Since $\pi^*\alpha$
vanishes on vertical vectors, $dv=dt$ on $\ker\pi_*$, so $dv|_E\neq0$, and
the kernel of $\operatorname{Hess}_qh_0$ is the $2$-dimensional space
$E\cap\ker dv$. It contains $T_q\Lambda$, because $v|_\Lambda=0$, hence
equals $T_q\Lambda$.

For (c), the points of $\Lambda$ over the branch
$s\approx0$ are $(\gamma(s),e^{\tau(s)+i\theta})$ with $\gamma(s)=(s,0,0,0,0)$
and, by~\eqref{eq:G-tau}, $r=\tau(s)-\operatorname{Re}G(\gamma(s))=0$; over the
branch $s\approx1$ we have $r=1$. This gives the description of $\Lambda$ near
$C_0$, where we have $v=r$, because both have the
same differential $dt-\pi^*\alpha$ (recall $d\operatorname{Re}G=\alpha$ on
$\mathbb{B}_{r_B}$), and vanish on $\Lambda$. Hence, $h_0=\rho+r^2$ there, with
$\rho=\sigma|y|^2+F_0(x)$ by Lemma~\ref{lem:rho}(d). Since $r(q)=0$ and
$F_0$ has Hessian $2x_1^2\,dx_2^2+2\sum_{j\ge3}dx_j^2$ at $(x_1,0,0,0,0)$, this
gives~\eqref{eq:hess-h0}. The covectors $dy_j$, $dx_2,\dots,dx_5$ and
$dr=dt-d\operatorname{Re}G$ are linearly independent, because $dt$ is
independent of $dx,dy$, and they vanish on $T_q\Lambda$, which is spanned by
$\partial_{x_1}+\partial_{x_1}\!\operatorname{Re}G\,\partial_t$ and the
rotation vector field of the fibre. Near $C_1$, use~\eqref{eq:G-tau} for $s$ near $1$.
\end{proof}

Thus the only degeneracy of $h_0$ sits in the direction $\partial_{x_2}$
along $C_0$ (respectively $\partial_{x_1}$ along $C_1$). Indeed,
$\{h_0=0\}$ contains, besides $\Lambda$, an additional locus
$\{z_1=z_3=z_4=z_5=0,\ z_2\in\mathbb{R},\ r=0\}$ near $C_0$. The goal of the next subsection
is to modify $h_0$ so that it vanishes precisely on $\Lambda$.

\subsection{Correcting the Hessian}\label{ssec:correction}

We remove this degeneracy by adding to $h_0$ small multiples of the
pluriharmonic functions $\operatorname{Re}z_2^2=x_2^2-y_2^2$ near $C_0$
and $\operatorname{Re}z_1^2$ near $C_1$, cut off away from these circles.

Fix $r_0>0$, and let
\[
N_0=\bigl\{|x_1|<3r_0,\ |y_1|<3r_0,\ |z_j|<3r_0\ (j=2,\dots,5),\ |r|<3r_0\bigr\},
\]
where $z=x+iy$ and $r=\log|w|-\operatorname{Re}G(z)$ are the functions of
Lemma~\ref{lem:h0}(c). For $r_0$ sufficiently small, $N_0$ is a
neighbourhood of $C_0$ contained in $W\cap\pi^{-1}(\mathbb{B}_{r_3})$, so that
$r$ is defined on $N_0$ and $\rho=\sigma|y|^2+F_0(x)$ there by
Lemma~\ref{lem:rho}(d), with $F_0$ as in~\eqref{e.F0}; moreover $v=r$ on $N_0$
by Lemma~\ref{lem:h0}(c), and
\[
\Lambda\cap N_0=\{|x_1|<3r_0,\ y_1=0,\ z_2=z_3=z_4=z_5=0,\ r=0\}.
\]
Define $N_1$ in the same way near $C_1$, interchanging $z_1$ and $z_2$ and
replacing $r$ by $r-1$. Since
$t=r+\operatorname{Re}G(z)$ is close to $0$ on $N_0$ and close to $1$ on
$N_1$, we have $N_0\cap N_1=\emptyset$ for $r_0$ small. Let
$\chi\in C^\infty(\mathbb{R},[0,1])$ satisfy $\chi=1$ on $[-r_0,r_0]$ and
$\operatorname{supp}\chi\subset(-2r_0,2r_0)$, and let
\[
\psi_0=\chi(x_1)\,\vartheta(y_1,r,z_2,z_3,z_4,z_5),
\]
where $\vartheta$ is a smooth function of $(y_1,r,z_2,\dots,z_5)$ with values in
$[0,1]$, equal to $1$ near $0$ and supported in
$\{|y_1|,|r|,|z_j|<2r_0\}$. Then $\psi_0\in C^\infty_c(N_0)$, and
$\psi_0=\chi(x_1)$ on $\Lambda\cap N_0$. Define $\psi_1\in C^\infty_c(N_1)$
in the same way, with the roles of $z_1$ and $z_2$ interchanged.
Extended by zero, $\psi_0$ and $\psi_1$ are smooth functions on $X$ with
disjoint compact supports. For $\kappa>0$ put
\[
g_\kappa=\tfrac{\kappa}{2}\bigl(\psi_0\operatorname{Re}z_2^2+\psi_1\operatorname{Re}z_1^2\bigr)\in C^\infty_c(X),
\qquad
h=h_0+g_\kappa\ \text{ on }W,
\qquad
\omega_\kappa=\widehat\omega+dd^cg_\kappa .
\]
Then $\omega_\kappa$ is a smooth, closed, real $(1,1)$-form on $X$,
cohomologous to $\widehat\omega$, and $\omega_\kappa=dd^ch$ on $W$.

\begin{lemma}\label{lem:correction}
For $\kappa>0$ sufficiently small the following hold.
\begin{enumerate}
\item[(a)] $\omega_\kappa$ is a Kähler form on $X$.
\item[(b)] $h|_\Lambda=0$, $dh|_\Lambda=0$, and $\operatorname{Hess}_qh$ is
  positive definite on $T_qX/T_q\Lambda$ for every $q\in\Lambda$.
\item[(c)] There are a neighbourhood $W''$ of $\Lambda$ and $c>0$ with
  $h\ge c\operatorname{dist}(\cdot,\Lambda)^2$ on $W''$.
\end{enumerate}
\end{lemma}

\begin{proof}
(a) Put $T=dd^c\bigl(\tfrac12\psi_0\operatorname{Re}z_2^2+\tfrac12\psi_1\operatorname{Re}z_1^2\bigr)$,
so that $dd^cg_\kappa=\kappa T$. This is a smooth real $(1,1)$-form on $X$
with support in the compact set
$S=\operatorname{supp}\psi_0\cup\operatorname{supp}\psi_1$. Fix a Hermitian
metric $g$ on $X$. Since $\widehat\omega$ is positive definite and $S$ is
compact, there is $m>0$ with $\widehat\omega\ge2m\,g$ on $S$, and there is
$C>0$ with $T\ge-C\,g$ on $S$. For $\kappa<m/C$ we get
$\omega_\kappa\ge m\,g$ on $S$, while $\omega_\kappa=\widehat\omega>0$ on
$X\setminus S$. Hence, $\omega_\kappa$ is a Kähler form.

(b) On $N_0$ the function $\operatorname{Re}z_2^2$ vanishes to second
order along $\{z_2=0\}\supset\Lambda\cap N_0$. Therefore,
$\psi_0\operatorname{Re}z_2^2$ and its differential vanish on
$\Lambda$, and for $q\in\Lambda\cap N_0$
\[
\operatorname{Hess}_q\bigl(\psi_0\operatorname{Re}z_2^2\bigr)
=\psi_0(q)\operatorname{Hess}_q\operatorname{Re}z_2^2
=\chi(x_1(q))\operatorname{Hess}_q(x_2^2-y_2^2).
\]
The same holds near $C_1$. Together with Lemma~\ref{lem:h0}(a) this gives
$h|_\Lambda=0$ and $dh|_\Lambda=0$. Since $h$ and $dh$ vanish along
$\Lambda$, we have $\operatorname{Hess}_qh(\xi,\eta)=0$ whenever
$\xi\in T_q\Lambda$, so $\operatorname{Hess}_qh$ descends to a quadratic form
on $T_qX/T_q\Lambda$.

Let $q\in\Lambda\cap N_0$ and write $x_1=x_1(q)$. Then $\psi_1=0$ near
$q$, and by~\eqref{eq:hess-h0}
\[
\operatorname{Hess}_qh=2\sigma\,dy_1^2+\bigl(2\sigma-\kappa\chi(x_1)\bigr)dy_2^2+2\sigma\sum_{j=3}^5dy_j^2
+\bigl(2x_1^2+\kappa\chi(x_1)\bigr)dx_2^2+2\sum_{j=3}^5dx_j^2+2\,dr^2 .
\]
By Lemma~\ref{lem:h0}(c), the ten covectors appearing here form a basis of the
dual of $T_qX/T_q\Lambda$. For $0<\kappa\le1$ all coefficients are at least
$\min(\kappa,2r_0^2)$: indeed,
$2x_1^2+\kappa\chi(x_1)\ge\kappa$ for $|x_1|\le r_0$ and $\ge2r_0^2$ for
$|x_1|\ge r_0$, and $2\sigma-\kappa\chi(x_1)\ge1$, because $\sigma\ge1$ and $\kappa\le1$. The situation near $C_1$
is symmetric. At the remaining points $q\in\Lambda\setminus(N_0\cup N_1)$,
the function $g_\kappa$ vanishes identically near $q$, and
$q\notin C_0\cup C_1$, so $\operatorname{Hess}_qh=\operatorname{Hess}_qh_0$ is
positive definite on $T_qX/T_q\Lambda$ by Lemma~\ref{lem:h0}(b).

(c) By (b) and the compactness of $\Lambda$, there is $c>0$ such that
$\operatorname{Hess}_qh\ge2c\,|\cdot|^2$ on the orthogonal complement of
$T_q\Lambda$ for every $q\in\Lambda$. Taylor's formula, applied uniformly on
the compact set $\Lambda$, gives $h\ge c\operatorname{dist}(\cdot,\Lambda)^2$
on a neighbourhood $W''$ of $\Lambda$.
\end{proof}

From now on $\kappa$ is fixed so that Lemma~\ref{lem:correction} holds, and we write
$\omega=\omega_\kappa$. Thus $\omega$ is a Kähler form on $X$, and $\omega=dd^ch$ on $W$.

\subsection{Proof of Theorem~\ref{thm:counterexample}}
Let $c>0$ be as in Lemma~\ref{lem:correction}(c). Since $\Lambda$ is compact,
after shrinking the neighbourhood $W''$ of that lemma we may assume that
$W''=\{\operatorname{dist}(\cdot,\Lambda)<r_1\}$ for some $r_1>0$ and that
$W''\subset W$, so that $h$ is defined on $W''$ and $\omega=dd^ch$ there.
We fix $0<\varepsilon<cr_1^2/4$ and put
\[
U=W'',\qquad h_\varepsilon=h-\varepsilon\colon U\to\mathbb{R},\qquad K=\{x\in U:\ h_\varepsilon(x)\le0\}.
\]
Then $K\subset\{\operatorname{dist}(\cdot,\Lambda)\le\sqrt{\varepsilon/c}\}\Subset U$
is compact, $\Lambda=\{h=0\}\subset\operatorname{int}K$, and $h_\varepsilon$
is strictly plurisubharmonic with $dd^ch_\varepsilon=\omega|_U$, where
$\omega$ is a Kähler form on $X$ by Lemma~\ref{lem:correction}(a).
This proves (i), and we rename $h_\varepsilon$ as $h$ in the
statement of the theorem. Decreasing $\varepsilon$ if necessary, we may also
assume that $\varepsilon$ is a regular value of $h$, so that $K$ is a compact
domain with smooth strictly pseudoconvex boundary $\{h=\varepsilon\}$.
Indeed, by Lemma~\ref{lem:correction}(b) the function $h$ vanishes to second
order along $\Lambda$ with transversally nondegenerate Hessian, so by Taylor's
formula $|dh|\ge c'\operatorname{dist}(\cdot,\Lambda)$ on a neighbourhood of
$\Lambda$ for some $c'>0$; hence $h$ has no critical points near $\Lambda$
outside $\Lambda$, while $\{h\le\varepsilon\}$ shrinks to $\Lambda$ as
$\varepsilon\to0$. The torus $\Lambda\subset\operatorname{int}K$ is smooth,
compact and totally real, as shown in Sec.~\ref{ssec:curve}.

\medskip
\noindent\textit{Isotropy and proof of (ii).} By Lemma~\ref{lem:correction}(b),
$dh=0$ at every point of $\Lambda$, as a covector on all of $T_qX$; this is
unchanged by the constant $\varepsilon$. Since $d^ch=-\tfrac14\,dh\circ J$,
the $1$-form $d^ch$ also vanishes on $T_qX$ for every $q\in\Lambda$, not merely on
$T_q\Lambda$. If $\iota\colon\Lambda\to X$ is the inclusion, then
$\iota^*d^ch=0$, and hence $\iota^*(dd^ch)=d(\iota^*d^ch)=0$. Thus:
{\it if a 2-form  $\Omega$ on $X$ coincides with $\lambda\,dd^c h$  near $\Lambda$ for some $\lambda\ge 0$
then $\Omega|_\Lambda=0$, i.e., $\iota^*\Omega = 0$.}
In particular $\omega|_\Lambda=0$, which proves (ii).

\medskip
For (iii) and (iv) we use the holomorphic annulus
\[
A=\{0\}\times\{1\le|w|\le e\},
\]
oriented as a complex curve, and the annulus
$B=\{(\gamma(s),e^{\tau(s)+i\theta}):0\le s\le1\}\subset\Lambda$, oriented
by $ds\wedge d\theta$. Orient $C_0$ and $C_1$ by increasing $\theta$. Then
$\partial A=C_1-C_0=\partial B$, so $\Sigma=A-B$ is a closed
(piecewise smooth) integral $2$-cycle in $X$.

\begin{lemma}\label{lem:period}
Suppose that $\Lambda$ is isotropic with respect to a Kähler form
$\Omega$ on $X$, i.e.,
$\Omega|_\Lambda=0$. Then $\int_\Sigma\Omega=\int_A\Omega>0$. In particular
$[\Omega]\ne0$ in $H^2(X,\mathbb{R})$.
\end{lemma}

\begin{proof}
Since $B\subset\Lambda$, we have $\int_B\Omega=0$. Also $\int_A\Omega>0$, because $A$ is a complex curve and
$\Omega$ is Kähler. Hence, $\langle[\Omega],[\Sigma]\rangle=\int_A\Omega-\int_B\Omega>0$.
\end{proof}

\noindent\textit{Proof of (iii).} Let $\lambda>0$, and let $\Omega$ be a Kähler form with
$\Omega=dd^c(\lambda h)$ near $K\supset\Lambda$. By the above discussion we have $\Omega|_\Lambda=0$,
so $[\Omega]\ne0$ by Lemma~\ref{lem:period}. By Lemma~\ref{lem:topX} we have $\mathcal H_X=0$,
so $\Omega$ is not a Hodge form. Taking $\lambda=N$ shows that $dd^c(Nh)$
has no Hodge extension. \qed

\medskip
\noindent\textit{Proof of (iv).} We first show that
$A\subset\widehat\Lambda_{\mathcal M(X)}$. Since $\partial A=C_0\cup C_1\subset\Lambda$,
it suffices to show that every complex hypersurface $Z$ of $X$ which meets
$A$ also meets $\Lambda$. Suppose, to the contrary, that $Z\cap A\ne\emptyset$
and $Z\cap\Lambda=\emptyset$.
We show that the intersection number of $Z$ with the cycle $\Sigma$ is
positive, while it must vanish because $\mathcal H_X=0$.

Let $Z$ be the zero set of a holomorphic section
$s$ of the holomorphic line bundle $L=[Z]$. Fix a
Hermitian metric on $L$, and let $c_1(L)$ be its Chern form. This is a closed
real $(1,1)$-form which represents the image of the Chern class of $L$ under
$H^2(X,\mathbb Z)\to H^2(X,\mathbb R)$ \cite[Ch.~1, \S1]{GH}. By
Lemma~\ref{lem:topX} this image is $0$, so
\[
  \int_\Sigma c_1(L)=0 .
\]
On the other hand, if $e$ is a local holomorphic frame of $L$ and $s=fe$, then
$c_1(L)=-\frac1\pi\,dd^c\log|e|^2$, and $\log|s|^2=\log|f|^2+\log|e|^2$ with
$\log|f|^2$ pluriharmonic away from $Z$. Therefore,
\[
  c_1(L)=-\tfrac1\pi\,dd^c\log|s|^2\qquad\text{on }X\setminus Z .
\]
Since $B\subset\Lambda\subset X\setminus Z$, Stokes' theorem gives
$\int_Bc_1(L)=-\frac1\pi\int_{\partial B}d^c\log|s|^2$. The annulus $A$ is
not contained in $Z$, because $\partial A\subset\Lambda$; hence
$A\cap Z=\{p_1,\dots,p_r\}$ is a finite nonempty set in the interior of $A$. Let $m_j\ge1$ be the order of vanishing of $s|_A$ at
$p_j$. By the Poincar\'e--Lelong formula on the curve $A$
\cite[Ch.~3, \S2]{GH}, which here reduces to the identity
$\frac1\pi\,dd^c\log|w-w_0|^2=\delta_{w_0}$ in one variable, we have
$\frac1\pi\,dd^c\log|s|^2=\sum_jm_j\delta_{p_j}-c_1(L)$ as currents on $A$,
and therefore
\[
  \int_Ac_1(L)=\sum_{j=1}^r m_j-\frac1\pi\int_{\partial A}d^c\log|s|^2 .
\]
Since $\partial A=\partial B$, we obtain
\[
  \int_\Sigma c_1(L)=\int_Ac_1(L)-\int_Bc_1(L)=\sum_{j=1}^r m_j\ge1 ,
\]
a contradiction. Hence, $A\subset\widehat\Lambda_{\mathcal M(X)}$.

Since $A\cap\Lambda=\Lambda\cap\pi^{-1}(0)=\partial A$ and $K$ shrinks to
$\Lambda$ as $\varepsilon\to0$, after decreasing $\varepsilon$ if necessary
we have $A\not\subset K$. Since $\Lambda\subset K$ implies
$\widehat\Lambda_{\mathcal M(X)}\subset\widehat K_{\mathcal M(X)}$, both hulls contain
points of $A\setminus K$. Therefore, neither $\Lambda$ nor $K$ is meromorphically
convex. This completes the proof of Theorem~\ref{thm:counterexample}. \qed

\medskip
Note that the fact that $K$ is not meromorphically convex also follows from
(iii) and the main theorem of~\cite{BGS}: a Hodge extension of $dd^ch$ would
contradict (iii).



\begin{thebibliography}{99}

\bibitem{Al} Alexander, H. \textit{Disks with boundaries in totally real and Lagrangian manifolds.} Duke Math. J. 100 (1999), 131--138.

\bibitem{AF} Andreotti, A., Frankel, T. \textit{The Lefschetz theorem on hyperplane sections.}
Ann. of Math. (2) 69 (1959), 713--717.

\bibitem{BGS} Boudreaux, B. J., Gupta, P., Shafikov, R. \textit{Hypersurface convexity and extension of Kähler forms.} Math. Z. 309 (2025), no. 1, Paper No. 15.

\bibitem{BS} Boudreaux, B. J., Shafikov, R. \textit{Meromorphic convexity on Stein manifolds.}
Indiana Univ. Math. J. 75 (2026), no. 1, 87--104.

\bibitem{CS} S-C Chen and M-C Shaw. Partial Differential Equations in Several Complex Variables. AMS/IP Studies in Advanced Mathematics, 2001.

\bibitem{CE} Cieliebak, K., Eliashberg, Y. \textit{New applications of symplectic topology in several complex variables.} J. Geom. Anal. 31 (2021), no. 3, 3252-3271.

\bibitem{Co} Col\c{t}oiu, M. \textit{On hulls of meromorphy and a class of Stein manifolds.} Ann. Scuola Norm. Sup. Pisa Cl. Sci. (4) 28 (1999), 405-412.

\bibitem{DG} Duval, J., Gayet, D. \textit{Riemann surfaces and totally real tori.} Comment. Math. Helv. 89 (2014), 299--312.

\bibitem{DS} Duval, J., Sibony, N. \textit{Polynomial convexity, rational convexity, and currents.} Duke Math. J. 79 (1995), no. 2, 487--513.

\bibitem{GH} Griffiths, P., Harris, J. \textit{Principles of algebraic geometry.}
Wiley-Interscience, New York, 1978.

\bibitem{Gu} Guedj, V. \textit{Approximation of currents on complex manifolds.} Math. Ann. 313 (1999), no. 3, 437--474.

\bibitem{Hatcher} Hatcher, A. \textit{Algebraic topology.} Cambridge University Press, Cambridge, 2002.

 \bibitem{Hormander} H\"ormander, L. \emph{An Introduction to Complex
   Analysis in Several Variables}, 3rd ed., North-Holland, 1990.
   
   \bibitem{Hu} Hulek, K. \textit{A remark on Runge approximation of meromorphic functions.} Ann. Scuola Norm. Sup. Pisa Cl. Sci. (4) 6 (1979), 185--191.
   
   \bibitem{Lee} Lee, J. M. \textit{Introduction to smooth manifolds.} 2nd ed.,
Graduate Texts in Mathematics 218, Springer, New York, 2013.

    \bibitem{Na} Narasimhan, R.  On the homology groups of Stein
   spaces, Invent. Math. 2 (1967), 377--385.
   
   \bibitem{Nem} Nemirovski, S. Finite unions of balls in $\mathbb C^n$ are rationally convex. Russ. Math. Surv. 63(2), 381-382
(2008).

    \bibitem{RS} Rourke, C. P., Sanderson, B. J. \emph{Introduction to
   Piecewise-Linear Topology}, Springer, Berlin, 1972.

\bibitem{Sch} Schwarz, G. \textit{Hodge decomposition---a method for solving boundary value problems.}
Lecture Notes in Mathematics 1607, Springer-Verlag, Berlin, 1995.

    \bibitem{Serre} Serre, J.-P.  Quelques probl\`emes globaux relatifs aux
   vari\'et\'es de Stein, Colloque sur les fonctions de plusieurs
   variables, Bruxelles, 1953, 57--68.
   
   \bibitem{St} Stout, E. L. \emph{Polynomial Convexity}. Progress in Mathematics 261,
Birkh\"auser, Boston, MA, 2007.

\end{thebibliography}
\end{document}